\documentclass[8pt]{article}
\usepackage{amssymb, fancyhdr,amsmath,amsthm,amsbsy,amsfonts, titletoc,  mathrsfs, mathabx, verbatim,extarrows, wasysym, euscript, enumerate, latexsym, mathtools, epsfig, subfigure, mathabx, graphicx, color, wrapfig, hyperref}

\topmargin - 1 in \overfullrule = 0pt
\theoremstyle{plain}

\newcommand{\beq}{\begin{equation}}
\newcommand{\eeq}{\end{equation}}
\newcommand{\bna}{\begin{eqnarray}}
\newcommand{\ena}{\end{eqnarray}}
\newcommand{\bea}{\begin{eqnarray*}}
\newcommand{\eea}{\end{eqnarray*}}

\newcommand{\h}{\hspace}
\newcommand{\normmm}[1]{{\left\vert\kern-0.25ex\left\vert\kern-0.25ex\left\vert #1 
    \right\vert\kern-0.25ex\right\vert\kern-0.25ex\right\vert}}
\newtheorem*{theorem 2.1}{Theorem 2.1}

\def\XXint#1#2#3{{\setbox0=\hbox{$#1{#2#3}{\int}$}
  \vcenter{\hbox{$#2#3$}}\kern-.5\wd0}}

\newcommand{\bh}{ \mathbf{h}}

\newcommand{\C}{{\mathbb C}} 
\newcommand{\R}{{\mathbb R}}
\newcommand{\bS}{{\mathbb S}}

\newcommand{\sF}{{\mathscr F}}

\newcommand{\rL}{{\mathrm L}}

\newcommand{\rR}{{\mathrm R}}

\newcommand{\rp}{{\mathrm p}}

\newcommand{\e}{\epsilon}

\newcommand{\al}{\alpha}

\newcommand{\p}{\partial}

\newcommand{\D}{\nabla}

\newcommand{\La}{\Delta} 
\newcommand{\Div}{\operatorname{div}}
\newcommand{\Ker}{\operatorname{Ker}}
\DeclareMathOperator{\Span}{span}

\newtheorem{thm}{Theorem}[section]
\newtheorem{lemma}[thm]{Lemma}
\newtheorem{cor}[thm]{Corollary}

\begin{document}
\title{}
\begin{center}  \textbf{ \large TWISTED SOLUTIONS TO A SIMPLIFIED ERICKSEN-LESLIE EQUATION} \end{center}
\begin{center} \textbf{YUAN CHEN, \h{2pt} SOOJUNG KIM, \h{2pt} YONG YU\footnote{The third author is partially supported by RGC grant 14306414.}} \end{center}
\begin{center}\textit{Department of Mathematics, The Chinese University of Hong Kong} \end{center}


\setcounter{section}{1}
\begin{abstract}
\noindent\textbf{ABSTRACT:} 
In this article we construct  global solutions to a  simplified Ericksen--Leslie system on  $\mathbb{R}^3$.  The   constructed solutions are twisted and periodic  along the $x_3$-axis with period $d = 2\pi \big/ \mu$. Here $\mu > 0$ is the twist rate. $d$ is the distance between two planes which are parallel to the $x_1x_2$\h{0.5pt}-\h{0.5pt}plane. Liquid crystal material is placed in the region enclosed by these two planes. Given a well-prepared initial data, our solutions exist  classically for all $t \in [\h{0.5pt}0, \infty)$.   However these solutions become  singular at all points on the $x_3$-axis and escape into third dimension exponentially while $t \rightarrow \infty$. An optimal blow up rate is also obtained.
  \end{abstract}\
\\

\setcounter{footnote}{1}

\noindent \textbf{I. INTRODUCTION}\\
\\
\textbf{I.1. BACKGROUND AND MOTIVATION} \hspace{4pt}\\

\noindent Ericksen--Leslie equation is a hydrodynamical system describing nematic liquid crystal flow. For the sake of simplifying and meanwhile preserving the energy dissipative property of the original Ericksen--Leslie equation, a simplified   version was proposed by Lin in \cite{L}.  With all the parameters in the system normalized to be $1$, the simplified equation can be read as follows: 
 \begin{equation}\label{eq-main-EL}
 \left\{
 \begin{aligned}
 & \p_t \phi +u\cdot \D \phi-\La \phi=|\h{0.5pt}\D\phi \h{0.5pt}|^2\phi\quad&\mbox{in $\R^3\times(0,\infty)$};\\
 &\p_t u+u\cdot \D u-\La u=-\D \rp-\D \cdot\big( \D \phi  \odot \D \phi\big)\quad&\mbox{in $\R^3\times(0,\infty)$};\\
 &\Div  {u}=0 \quad&\mbox{in $\R^3\times(0,\infty)$}.
      \end{aligned}\right.
   \end{equation}
 In \eqref{eq-main-EL},  $\phi :\R^3\times(0,\infty)\to  \mathbb{S}^2$ represents the macroscopic orientation of a nematic liquid crystal.  $u: \R^3\times(0,\infty)\to\R^3$ is the   velocity field  of  fluid. $\rp:\R^3\times(0,\infty) \to\R$ is the pressure. The    stress tensor $\D \phi  \odot \D \phi$ is defined  with its 
$(i,j)$\h{0.5pt}-\h{0.5pt}th entry given by 
\begin{eqnarray*} 
\big(\D \phi  \odot \D \phi \big)_{ij}= \big< \p_i\phi,\, \p_j\phi\big>. 
\end{eqnarray*}
Here $\big<\cdot, \cdot\big>$ is the standard inner product  on $\mathbb{R}^3$. As one can see, the system \eqref{eq-main-EL} is   a coupled system consisting of the non-homogeneous incompressible Navier-Stokes equation and the transported heat flow of harmonic map.

Many research works have been devoted to the study of  \eqref{eq-main-EL}. We refer readers to the survey article \cite{LW} by Lin-Wang and references therein.  Most recently solutions of  \eqref{eq-main-EL} with finite time singularity have also  been constructed by authors in  \cite{HLLW}, where the spatial domain is a bounded open set in $\mathbb{R}^3$.   In contrast to \cite{HLLW}, in the current work, we are concerned with global solutions of (1.1) which  become  singular at $t = \infty$. Our motivation originates from  the twisted ansatz for nematic liquid crystal in \cite{St}. In fact our solutions are supposed to admit a special form which is given by
\begin{equation}\label{eq-form-sol}
u=   \mathscr{V}(x_1, x_2, t ) \h{20pt} \text{and}\h{20pt} 
\phi = e^{\h{0.5pt}\mu \h{0.5pt} x_3  \mathrm{R}} \h{1pt} \psi(x_1, x_2, t).
\end{equation}
Here   $\mu > 0$ is called twist   rate of nematic liquid crystal. $\mathscr{V}$ and $\psi$ are two 3-vectors. Particularly $\psi$ takes its value in $\mathbb{S}^2$. $\mathrm{R}$ in \eqref{eq-form-sol} denotes  the generator of horizontal rotations, which can be represented by 
   \begin{equation*}\displaystyle
\rR=  
 \left(
\begin{array}{ccc}
0&-1&0\\ 
1&0&0\\
0&0&0
\end{array}
\right).
\end{equation*}For a given real number $\alpha$, the exponential matrix of $\alpha \mathrm{R}$ equals to \begin{eqnarray*}  \left(
\begin{array}{ccc}
\cos \al&-\sin \al&0\\ 
\sin \al&\cos\al&0\\
0&0&1
\end{array}
\right).\end{eqnarray*}
Plugging the ansatz \eqref{eq-form-sol} into \eqref{eq-main-EL}, we obtain the equation satisfied by $ \mathscr{V}$ and $\psi$ as follows:  
\begin{gather}\label{eq-main}
 \left\{
 \begin{aligned}
 & \frac{D}{D t} \psi  -\La_2 \psi=|\h{0.5pt}\D^h \psi \h{0.5pt}|^2\psi - \mu \h{0.5pt} \mathscr{V}_3 \h{0.5pt} \mathrm{R} \psi + \mu^2 \h{0.5pt} \bigg[ \h{1pt}\mathrm{R}^2 \psi + | \h{0.5pt}\mathrm{R} \psi \h{0.5pt} |^2 \psi \h{1pt}\bigg] \quad&\mbox{in $\R^2\times(0,\infty)$};\\[3pt]
 &\frac{D}{D t} \mathscr{V}^h  -\La_2 \mathscr{V}^h  =- \D^h\mathrm q-\D^h \cdot\big( \D^h \psi  \odot \D^h \psi\big)\quad&\mbox{in $\R^2\times(0,\infty)$};\\[3pt]
 &\frac{D}{D t} \mathscr{V}_3  -\La_2 \mathscr{V}_3  =   -\mu \h{0.5pt} \big< \h{0.5pt}\La_2 \psi, \mathrm{R} \psi \h{0.5pt}\big>\quad&\mbox{in $\R^2\times(0,\infty)$},
      \end{aligned}\right.
   \end{gather}
   where $\mathscr{V}^h$ satisfies the incompressibility condition: 
   \begin{equation}\label{eq-div-free-ho-velocity}
 \D^h \cdot \mathscr{V}^h  = 0 \quad \mbox{in $\R^2\times(0,\infty)$}.
\end{equation}
In  \eqref{eq-main}-\eqref{eq-div-free-ho-velocity}, $\D^h$ and $\La_2$ are the gradient and  Laplace operators on $\mathbb{R}^2$, respectively. $\mathscr{V}^h = (\mathscr{V}_1, \mathscr{V}_2)^t$ is the horizontal part of the vector field $\mathscr{V}$. $\mathscr{V}_3$ is the third component of $\mathscr{V}$.  The differential operator $D/ Dt $ is the material derivative $\p_t + \mathscr{V}^h \cdot \D^h$.  

When $  \mu = 0$ and $\mathscr{V}_3 \equiv 0$, the system \eqref{eq-main}-\eqref{eq-div-free-ho-velocity} is then reduced to  a 2D version of \eqref{eq-main-EL}. Its vorticity formulation has been studied by authors in \cite{CY1}. In fact a global weak solution is obtained in [3] under the assumption that the initial vorticity of fluid lies in $\mathrm{L}^1(\mathbb{R}^2)$ and the initial director field  has finite Dirichlet energy. Furthermore in \cite{CY2}, when initial vorticity and director field are sufficiently close to  \begin{eqnarray}\label{sta} \left(0, \h{1pt}e^{m \theta \mathrm{R} + \alpha_0 \mathrm{R}} \h{1pt} \mathbf{h}\left(\frac{r}{\sigma_0} \right) \right)\end{eqnarray} in some norm space, then (1.1) in 2D admits
a global classical solution which has the form
\begin{eqnarray*}\mathscr{V}^h = \frac{f(r,t)}{  r^2}
\left(
\begin{array}{r}
- x_2\\ 
x_1
\end{array}
\right) \h{20pt} \text{and}\h{20pt}\psi = e^{m \theta \mathrm{R}} \h{1pt} {\psi}_*(r,t).
\end{eqnarray*}
Here $(r, \theta)$ is the polar coordinate on $\mathbb{R}^2$. $m$ is an integer with $|m| \geq 4$. $\sigma_0 > 0$ and $\alpha_0 \in \mathbb{R}$ are two constants.  For any $\rho > 0$,  $\mathbf{h}$ is a 3-vector defined by 
\begin{equation}\label{equiv-harmonic}
\mathbf{ h}(\rho)= \left(
\begin{array}{c}
\mathbf{h}_1(\rho)\\ 
0\\
\mathbf{h}_3(\rho)
\end{array}
\right),\quad \text{where}\h{5pt}\mathbf{h}_1(\rho)=\frac{2}{\rho^{|m|}+\rho^{-|m|}}\quad\text{and}\quad \mathbf{h}_3(\rho)=\frac{\rho^{|m|}-\rho^{-|m|}}{\rho^{|m|}+\rho^{-|m|}}.
 \end{equation}The authors in \cite{CY2} also show that when the absolute value of the initial circulation Reynolds number $\omega$  is suitably small, then there exists a positive constant $\sigma_{\infty}$ and an angular function $\alpha(t)$ so that \begin{eqnarray*}   
\psi \sim e^{ \h{0.5pt} m  \theta \h{0.5pt} \mathrm{R} + \alpha(t) \h{0.5pt}  \mathrm{R}} \, \mathbf{h}\left(\frac{r}{ \sigma_\infty} \h{0.5pt}\right), \h{20pt}\text{as $t \rightarrow \infty$.}
\end{eqnarray*}
 Moreover for some $\alpha_{\infty} \in \mathbb{R}$, the following limit holds for  the angular function $\alpha(t)$: \begin{eqnarray*} \lim_{t \rightarrow \infty} \big[ \h{0.5pt} 4 \h{1pt} \alpha(t) +  m  \h{0.5pt}\omega \h{0.5pt} \log t  \h{1pt}\big] = \alpha_{\infty}. \end{eqnarray*} The above limit implies that the director field keeps rotating around the $x_3$-axis as $t $ tends to $ \infty$, provided that $|\h{1pt}\omega \h{0.5pt}| $ is small and nonzero. It can be easily checked that (1.5) gives  a stationary solution to the vorticity formulation of (1.1) in 2D. The results in [4] then indicate that these stationary solutions are globally dynamically unstable in the space $\mathrm{L}^1(\mathbb{R}^2) \times \mathrm{H}_{\mathbf{e}_3}^1(\mathbb{R}^2; \mathbb{S}^2)$, where $\mathbf{e}_3 \in \mathbb{S}^2$ is the north pole.\\
 \\
\textbf{I.2. MAIN RESULTS}\\
\\
Inspired by the works \cite{CY1}-\cite{CY2}, we construct in this article a global solution  to  \eqref{eq-main}-\eqref{eq-div-free-ho-velocity}
  under the following  $m$-equivariant   ansatz:
\begin{equation}\label{eq-velocity-form}
\displaystyle
\mathscr{V} =  
 \frac{ {W}(r,t)}{  r^2}
\left(
\begin{array}{c}
- x_2\\ 
x_1\\
0
\end{array}
\right)+ 
\left(
\begin{array}{c}
0\\ 
0\\
 {V}( r,t)
\end{array}
\right) \h{20pt}\text{and}\h{20pt} 
\psi = e^{\h{1pt} m \h{0.5pt} \theta  \h{0.5pt} \rR} \h{1pt}  \varphi( r ,t) .
\end{equation} 
Here $ {W}$, $ {V}$ are two real-valued functions  and $\varphi$ is an $\mathbb{S}^2$-valued vector field. They all depend on the variables $t$ and $r$ only.  By this ansatz, the incompressibility condition (1.4) is automatically satisfied. Meanwhile the system \eqref{eq-main} is  reduced to  
   \begin{equation}\label{eq-main-phi-W-V}
 \left\{
 \begin{aligned}
     & \p_t\h{0.5pt} \varphi   + \left( \frac{m  {W}}{r^2} + \mu \h{0.5pt} {V} \right) \rR \varphi \h{2pt} = \h{2pt} \Delta_2\h{0.5pt} \varphi  + \big\vert \h{1.5pt} \p_{r}\h{0.5pt}\varphi\h{0.5pt}\big\vert^2 \varphi    + \left( \frac{m^2}{r^2} + \mu^2 \right) \bigg[\rR^2   \varphi+ \big\vert \h{1.5pt} \rR  \varphi \h{0.5pt}\big\vert^2  \varphi\bigg] \quad&\mbox{in $\R^2\times(0,\infty)$};\\[3pt]
     & \p_t  {W}  = \p_{rr}  {W}  - \frac{1}{r} \p_{r}  {W}  -  {m} \h{0.5pt}\Big\langle \Delta_2 \varphi  , \, \rR \h{0.5pt} \varphi\Big\rangle  &\mbox{in $\R^2\times(0,\infty)$};\\[3pt]
     & \p_t    {V}  = \Delta_2  {V}     - \mu  \h{0.5pt}\Big\langle \Delta_2  \varphi , \, \rR \h{0.5pt}  \varphi\Big\rangle \quad&\mbox{in $\R^2\times(0,\infty)$}.
    \end{aligned}\right.
   \end{equation}
If  $\left(\h{0.5pt} W, V, \varphi \h{0.5pt}\right)$  is  a global solution of    \eqref{eq-main-phi-W-V},   then  it provides      
  a global  solution $\left( u, \phi\right)$ to \eqref{eq-main-EL} by the change of variables in    \eqref{eq-form-sol} and \eqref{eq-velocity-form}.

 Before stating the main theorem, we  introduce   some  notations. The map  $\mathbf{ h}$   in \eqref{equiv-harmonic}   generates  a 2-parameter family of $m$-equivariant harmonic maps in 2D: 
  $$\displaystyle \Big\{ \h{1pt} e^{m\theta\rR} \h{0.5pt}\mathbf{ h}^{\al,\sigma}(r)\,:\,    \al\in \mathbb{R}\h{5pt} \text{and}\, \h{5pt} \sigma>0\h{1pt}\Big\}.$$  Here we simply call a map  harmonic if it is a harmonic map from $\mathbb{R}^2$ to $\mathbb{S}^2$. For any $\alpha \in \mathbb{R}$ and $\sigma \in \mathbb{R}^+$, $\mathbf{h}^{\alpha, \sigma}$ in the set given above is defined by $$  \mathbf{ h}^{\al,\sigma}(r):= e^{ \al \rR} \h{1pt}\mathbf{h}\left(\frac{r}{\sigma}\right), \h{30pt}\forall\h{2pt}r > 0.$$  It  satisfies the boundary conditions   $\mathbf{ h^{\al,\,\sigma}}(0)=-\mathbf{ e_3}$ and  $\mathbf{ h^{\al,\,\sigma}}(\infty)=\mathbf{ e_3}$. Moreover $\mathbf{h}^{\alpha, \sigma}$ attains the minimal Dirichlet  energy $4\pi |m|$  in the class $\Sigma_m$ of all $m$-equivariant maps.  
 Here $\Sigma_m$ is given as follows:
    \begin{equation*}
\Sigma_m:=\Big\{\, \psi\,:\, \R^2\to \bS^2  \,\left\vert \,\psi=e^{m\theta \rR} \h{1pt} \varphi(r),\  \left\| \h{1pt}\D \psi\h{1pt}\right\|_{\rL^2(\R^2)}<\infty, \right.\,\varphi(0)=-\mathbf{ e_3},\, \varphi(\infty)=\mathbf{ e_3} \,\Big\}.
\end{equation*}
Associated with an $m$-equivariant map $\psi =e^{m\theta \rR} \h{1pt} \varphi( r)\in \Sigma_m$, we define 
 \[
q\,=\,  {q}\h{1pt} [\h{1pt}\varphi\h{1pt}]:= \p_{r} \varphi-\frac{{|m|}}{r}\, \varphi\times \rR \varphi.
 \]
 The tangent vector   $ {q}= {q} \h{0.5pt}[\h{0.5pt}\varphi \h{0.5pt}]\in T_{\varphi}\bS^2$ provides us with the information on the harmonicity of $\psi = e^{m \theta \mathrm{R}} \h{1pt}\varphi$. In fact if $q\h{0.5pt}[\h{0.5pt} \varphi \h{0.5pt}] = 0$, then the vector field $\psi$ is a harmonic map in 2D. More properties on the tangent vector field $q \h{0.5pt}[\h{0.5pt} \varphi \h{0.5pt}]$ can be found in  \cite{GKT}. We also study in this article the fluid with possibly nonzero circulation Reynolds number. Therefore we should introduce the Oseen part of the variable $W$ (see \eqref{eq-velocity-form}). In fact we denote  by $ {W}^{os}$  the Oseen part of $ {W}$, which is explicitly  given by 
 \begin{equation}\label{eq-def-Oseen}
   {W}^{os}( r,t) = \omega \left( 1 -  e^{ - r^2 \big/ l(t)} \right) \h{10pt}\text{with  \h{2pt}$l(t) :=  4 \h{0.5pt} t + r_0^2$.} 
\end{equation}
Here $r_0 > 0$ is the initial core radius of the Oseen vortex.  $\omega \in \mathbb{R}$ is the circulation Reynolds number. When $t = 0$, we simply denote by $W_{in}^{os}$ the function $W^{os}(\cdot, 0)$. In the following statements $A\, \lesssim \,B$ means that there is a   constant $c>0$ so that $A \leq c B$. 
Here  $c$ depends possibly on the structural parameters  $m,$  $ \mu , $  $  \omega , $ and $r_0$. Throughout the article for any $ p \in [1, \infty]$, we use $\rL^p$ to denote the functional spaces $\mathrm{L}^p\left(r\h{0.5pt}\mathrm d \h{0.5pt} r\right)$. Usually if  integrand in an integration is a function of variable $r$, then we suppress the notation   $r\h{0.5pt}\mathrm d \h{0.5pt} r$ from the  integration and simply employ  the following agreement:
$$ \int_0^\infty f \,= \,\int_0^\infty  f\h{2pt}r\h{0.5pt}\mathrm d \h{0.5pt} r.$$ 
We also need an $X$-space, which  is   defined by
\begin{equation}\label{def-space-X}
X :=\Big\{ \h{1pt}z : [\h{0.5pt}0,\infty)\to \C\h{4pt}\Big\vert\h{4pt}\|\h{1pt}z\h{1pt}\|_X < \infty \h{1pt}\Big\}.
\end{equation}  Here $\| \cdot \|_X$ is the following norm for functions in the space $X$: \begin{eqnarray}\label{Xnorm}\left\|\h{0.5pt} z\h{0.5pt} \right\|_X^2:= \int_0^{\infty} \left( \big|\h{1pt}\p_\rho z\h{1pt}\big|^2 + \frac{|\h{1pt}z\h{1pt}|^2}{\rho^2} \right) \h{0.5pt} \rho \h{1pt} \mathrm d\h{1pt}\rho. \end{eqnarray}Now we give our main results in this article.

  \begin{thm}\label{thm-large-time-behavior} 
  Suppose that $m $ is an integer satisfying  $|m| \geq3$.  $\mu > 0 $ is a given twist rate. Suppose that the initial velocity field of fluid satisfies  $V_{in} \in \mathrm{L}^2$ and \begin{eqnarray*}W_{in} =W^{os}_{in}+W_{in}^*  \h{20pt}\text{with}\h{10pt} \frac{W_{in}^*}{r} \in \mathrm{L}^2. \end{eqnarray*} Given two constants $\Theta_{in}\in\R$, $\sigma_{in}>0$ and two single-variable functions $z_{j, in}$ ($j = 1, 2 $) on $\mathbb{R}^+$ with  \begin{eqnarray}\label{eq-in-orthogonal-z-h_1}
  \begin{aligned}
	\|\h{1pt}z_{j, in} \h{1pt}\|_{\mathrm{L}^{\infty}} < 1/2 \h{20pt}\text{and}\h{20pt}\int_0^\infty \h{1.5pt} z_{j, in}     \h{1pt}  \mathbf h_1   =0 ,
    \end{aligned}
    \end{eqnarray} we assume  the following representation  for the $\mathbb{S}^2$-valued vector field $\varphi_{in}$:  \begin{eqnarray}\label{eq-varphi-decomposed-initial}
  \begin{aligned}
    \varphi_{in}(r)&= {e^{\Theta_{in}\h{1pt} \rR}} \Big\{\, \mathbf{ h} (\rho)+ \gamma_{in}   (\rho) \h{1pt} \mathbf{ h}(\rho)+ z_{1, in}(\rho) \h{1pt}\mathbf{e}_2+ z_{2, in}(\rho) \h{1pt} \mathbf{ h}(\rho)\times \mathbf{e}_2\,\Big\}, \qquad\hbox{where}\quad \rho=\frac{ r}{\sigma_{in}}.\,\\[2pt]
   \end{aligned}
    \end{eqnarray}
Here the vector field $\mathbf{h}$ is defined in \eqref{equiv-harmonic}. $\mathbf{e}_2$ denotes the unit vector $(0,1,0)^t$. The function $\gamma_{in}(\cdot) $ is   given by  \begin{eqnarray*} \gamma_{in} = \Big( 1 - | \h{0.5pt}z_{in} \h{0.5pt}|^2 \Big)^{1/2} - 1 \h{20pt}\text{with $z_{in} = z_{1, in} + i \h{0.5pt}z_{2, in}$.}
\end{eqnarray*}
For any given $\varepsilon \in(0,1),$  there exists a  positive constant  $\delta_*$ (depending on $m,$ $\mu,$ $\omega$, $ r_0$ and $ \varepsilon$) such that if  
  \begin{equation}\label{eqq-assump-small-initial}
  \begin{aligned}
  \big\|\h{1pt}  z _{in}\h{1pt} \big\|_{X }+\|\h{1pt} V_{in}\h{1pt} \|_{\rL^2}+  \left\|\h{1pt} \frac{W^*_{in}}{r}\h{1pt} \right\|_{\rL^2 } + \sigma_{in} + \left(     \int_0^{\infty} |\h{1pt}z_{in}(\rho) \h{1pt}|^2 \h{1pt}\rho\h{1.5pt}\mathrm{d}\h{0.5pt}\rho \right)^{1/2}\,\, <\,\, \delta_*,
     \end{aligned}
    \end{equation}
then the followings hold  for the equation \eqref{eq-main-phi-W-V}:\\
\\
(i). Supplied with the initial data $\big(W_{in}, V_{in}, \varphi_{in}\big)$, \eqref{eq-main-phi-W-V} admits a classical solution, denoted by $\big(W, V, \varphi\big)$, on the time interval $[\h{0.5pt}0, \infty)$. For some  $\mathrm{C}^1$-regular time dependent parameter functions  $\big(\sigma(t)\h{0.5pt},\h{0.5pt} \Theta(t)\big) $,       the vector field $\varphi$     can be expressed    as 
  \begin{eqnarray}\label{eq-decom-varphi}
  \varphi(r, t) =e^{\Theta(t) \,\rR}\Big\{\,\mathbf{ h}(\rho)+ \gamma   (\rho,t) \h{1pt}\mathbf{ h}(\rho)+ z_1(\rho,t) \h{1pt}\mathbf{e}_2 + z_2( \rho,t) \h{1pt}\mathbf{ h}(\rho)\times \mathbf{ e}_2 \,\Big\}, \quad\hbox{where} \h{5pt} \rho=\frac{ r}{\sigma(t)}\,. 
    \end{eqnarray}
For a fixed $t > 0$, $z_1(\rho,t)$ and $z_2(\rho,t)$ in the above expression are two functions in the space $X \cap \mathrm{L}^2(\rho\h{1pt}\mathrm{d} \h{0.5pt}\rho)$.  $\gamma$ is a function given in \eqref{eq-def-gamma}.  Let $z$ be the complexified function $z_1 + i \h{1pt}z_2$.  Then for all $t > 0$, $z(\cdot, t)$ satisfies the  orthogonal condition:  \begin{eqnarray}\label{eq-orthogonal-z-h_1}
  \begin{aligned}
	 \int_0^\infty \h{1.5pt} z(\rho, t)     \h{1pt}  \mathbf h_1(\rho)  \h{1.5pt} \rho \h{1pt} \mathrm d\h{0.5pt}\rho  =0.
    \end{aligned}
    \end{eqnarray}   Moreover the  following estimate holds for the $X$-norm of $z$: \begin{eqnarray}\label{eq-int-L^2-q}
 \int_0^\infty \exp \left\{\frac{2\mu^2}{m^2}\h{0.5pt}s \right\} \h{1pt} \| \h{1pt}z(\cdot, s)\h{1pt}\|_{X}^2 \h{1pt} \mathrm d\h{0.5pt}s\,\, \lesssim\,\, 1;
\end{eqnarray}
(ii). The functions $W$ and $V$ can be decomposed into \begin{eqnarray*} W = W^{os} + W_1^* + W_2^* \h{20pt}\text{and}\h{20pt}V = V_1 + V_2,
\end{eqnarray*}respectively. Moreover for all $t > 0$,  $V_1$, $V_2$, $W_1^*$, $W_2^*$  satisfy the following time decay estimates: \begin{eqnarray}\label{eq-est-time-decay}
\begin{aligned} 
\left\| \h{1pt} V_{1}  \h{1pt}\right\|_{\mathrm{L}^{\infty}}^2 + \Big\| \h{1pt} W^*_{1}\big/ r^2 \h{1pt}\Big\|_{\mathrm{L}^{\infty}}    \h{2pt}\lesssim\h{2pt} t^{- 1} , \h{20pt}
 \left\| \h{1pt} V_2 \h{1pt}\right\|_{\mathrm{L}^{2}}+  \left\| \h{1pt}   W^*_{2}\big/{ r} \h{1pt}\right\|_{\mathrm{L}^{2}}  \h{2pt}
 \lesssim  \,(1+t)^{-1};
\end{aligned}
\end{eqnarray}
(iii). The scaling function $\sigma(\cdot)$ decays exponentially. More precisely it holds 
 \begin{equation}\label{eq-sigma-exponential}
\left(1-\varepsilon\right) \h{1pt}e^{ - \frac{\mu^2}{m^2}\h{1pt}t } \,\sigma_{in} \h{2pt}\leq \h{2pt}\sigma\h{2pt} \leq\h{2pt}  \left(1+\varepsilon\right)  \h{1pt}   e^{ - \frac{\mu^2}{m^2}\h{1pt}t } \,\sigma_{in}   \qquad\hbox{for all $t\geq0$}.
\end{equation}  It is this estimate that gives us the blow-up of $\varphi$ at $t = \infty$.
  \end{thm}\
 \\
\textbf{I.3. SOME REMARKS ON THEOREM 1.1}\\
 \\
 We would like to point out three remarks on Theorem 1.1. \\
 \\
 \textit{I. Motivation for the initial vector field $\varphi_{in}$ in \eqref{eq-varphi-decomposed-initial}. }
 \\
 \\ 
 If we decouple the fluid part from the system \eqref{eq-main-phi-W-V}, then the first equation in \eqref{eq-main-phi-W-V} gives us the following heat flow of harmonic maps: \begin{eqnarray}\label{har-3d} 
 \p_t\h{0.5pt} \varphi    \h{2pt} = \h{2pt} \Delta_2\h{0.5pt} \varphi  + \big\vert \h{1.5pt} \p_{r}\h{0.5pt}\varphi\h{0.5pt}\big\vert^2 \varphi    + \left( \frac{m^2}{r^2} + \mu^2 \right) \bigg[\rR^2   \varphi+ \big\vert \h{1.5pt} \rR  \varphi \h{0.5pt}\big\vert^2  \varphi\bigg], \qquad\mbox{in $\R^2\times(0,\infty)$}.
\end{eqnarray}In \eqref{eq-varphi-decomposed-initial} we choose $\varphi_{in}$ to be a small perturbation of the harmonic map $\mathbf{h}^{\Theta_{in}, \h{0.5pt}\sigma_{in}}$ in some norm space. Obviously $\mathbf{h}^{\Theta_{in}, \h{0.5pt}\sigma_{in}}$ is not a stationary solution to the equation \eqref{har-3d} if $\mu \neq 0$. However by applying the following change of variables: \begin{eqnarray}\label{chan-var} \varphi(r,t)\,=\,  \Phi(y,s), \qquad\hbox{where}  \quad y\h{0.5pt} =\h{0.5pt} \frac{r}{\lambda(t)},\quad s(t) \h{0.5pt} =\h{0.5pt} \int_0^t  \lambda^{-2}(\tau)\h{1pt}\mathrm{d} \h{0.5pt}\tau \quad\text{and}\quad \lambda(t) = \mu^{-1} e^{ - \mu^2 t \big/ m^2} , \end{eqnarray}the equation \eqref{har-3d}  can be rewritten as 
\begin{equation}\label{eq-Phi}
 \p_s\h{0.5pt} \Phi   + \frac{ y \h{1pt}\p_y\Phi}{2s + m^2} \h{2pt} = \h{2pt} \Delta_2 \Phi   + \big\vert \h{1.5pt} \p_{y}\h{0.5pt}\Phi\h{0.5pt}\big\vert^2 \Phi    + \left( \frac{m^2}{y^2} + \frac{m^2}{2s + m^2} \right) \bigg[\rR^2   \Phi+ \big\vert \h{1.5pt} \rR  \Phi \h{0.5pt}\big\vert^2  \Phi\bigg]  .
\end{equation}Formally if we take $s\to\infty$, then  a global solution $\Phi$ of the above equation should asymptotically approach to a solution of the equation 
\begin{equation}\label{eq-limit-eq}
   \Delta_2 \Phi   + \big\vert \h{1.5pt} \p_{y}\h{0.5pt}\Phi\h{0.5pt}\big\vert^2 \Phi    +  \frac{m^2}{y^2}   \bigg[\rR^2   \Phi+ \big\vert \h{1.5pt} \rR  \Phi \h{0.5pt}\big\vert^2  \Phi\bigg] =0.
\end{equation}This observation motivates us the choice of $\varphi_{in}$ in \eqref{eq-varphi-decomposed-initial}.\\
\\
\textit{II. Results on the pure harmonic map heat flow.}\\
\\
With slight modifications to the proof of Theorem 1.1, we have the following results for \eqref{har-3d}, the heat flow of harmonic map: \begin{cor} 
 \label{cor-main-result}
  Let $m $ be an integer satisfying  $|m| \geq3$. $\mu > 0$ is a twist rate. Suppose that  \eqref{eq-in-orthogonal-z-h_1}-\eqref{eq-varphi-decomposed-initial} hold for the initial vector field $\varphi_{in}$. 
For any $\varepsilon \in(0,1),$  there exists a  positive constant  $\delta_*$ (depending on $m,$ $   \mu $ and $ \varepsilon$) such that if   \begin{equation*} 
      \begin{aligned}
  \big\|\h{1pt}  {z} _{in}\h{1pt} \big\|_{X} + \sigma_{in} + \left(     \int_0^{\infty} \big|\h{1pt}z_{in}(\rho) \h{1pt}\big|^2 \h{1pt}\rho\h{1.5pt}\mathrm{d}\h{0.5pt}\rho \right)^{1/2}\,\, <\,\, \delta_*,
     \end{aligned}
    \end{equation*}
then the followings hold  for  the equation  \eqref{har-3d}: \\
\\
(i). Supplied with the initial data $  \varphi_{in} $, \eqref{har-3d} admits a classical solution, denoted by $ \varphi $, on the interval $[\h{0.5pt}0, \infty)$. For some  $\mathrm{C}^1$-regular time dependent parameter functions  $\big(\sigma(t)\h{0.5pt},\h{0.5pt} \Theta(t)\big) $,       the vector field $\varphi$     can be expressed  in terms of \eqref{eq-decom-varphi}. Moreover the functions $z_1$ and $z_2$ in \eqref{eq-decom-varphi} satisfies \eqref{eq-orthogonal-z-h_1}-\eqref{eq-int-L^2-q};\\
\\
(ii).\h{0.5pt}The parameter function $\sigma(\cdot)$ decays to $0$ as $t \rightarrow \infty$. The optimal decay rate is given in  \eqref{eq-sigma-exponential}. Furthermore there exists a constant $\Theta_\infty\in\R$ such that $\Theta(t)\to \Theta_\infty$ as $t\to \infty$. 
\end{cor}\noindent For the simplified Ericksen-Leslie equation, we only have algebraic decay for   $V_1$, $V_2$, $W_1^*$ and $W_2^*$ in \eqref{eq-est-time-decay}. It is not enough to show the $\mathrm{L}^1$-integrability of $\Theta'$ on $\mathbb{R}^+$. But for the pure heat flow of harmonic maps, we do have $\mathrm{L}^1$-integrability of $\Theta'$ on $\mathbb{R}^+$, which gives us the convergence of $\Theta(\cdot)$ in (ii) of Corollary 1.2. \\
\\
\textit{III. Comparison with some known results.} 
\\
\\
Given $m$ an integer satisfying $|m|\geq3$ and $\mu > 0$ a twist rate, our  global  solution $\big( W, V, \varphi \big)$ obtained by Theorem 1.1 gives a global solution $(u, \phi)$ to \eqref{eq-main-EL} through the change of variables in   \eqref{eq-form-sol} and \eqref{eq-velocity-form}.  $u$ is homogeneous in terms of the variable $x_3$, while $\phi$    is twisted and periodic  along the $x_3$-axis with period $2\pi \big/ \mu $. Moreover the director field  $\phi$ blows up with an exponential rate at all points on the $x_3$-axis,  as $t$ tends to $\infty$. By the representation of $\varphi$ in \eqref{eq-decom-varphi},  except at $r = 0$ where $\varphi = - \mathbf{e}_3$, for all $r > 0$, $\varphi(r, t) $ converges to $\mathbf{e}_3$ as $t \rightarrow \infty$. In other words our solution $\phi$ escapes into third dimension for large time $t$.  In \cite{CY2} the authors show that for $\mu = 0$, $V \equiv 0$ and $| \h{0.5pt}\omega \h{0.5pt} |$  suitably small, where $\omega $ is the circulation Reynolds number, the vector field $\varphi$ should keep rotating around the $x_3$-axis as $t \rightarrow \infty$. However our results in Theorem 1.1 indicate that for $\mu > 0$, the associated vector field $\varphi$ should escape into third dimension exponentially. Compared with the blow-up of $\varphi$ in Theorem 1.1, the oscillating effect from the  swirling velocity field $u$ can be ignored. Even when $\mu > 0$ is small, the system \eqref{eq-main-phi-W-V} should not be regarded as a perturbed system of the one with $\mu = 0$ (the case studied in \cite{CY2}). Interesting reads should also refer to \cite{GGT,GNT}, where global solutions for harmonic map heat flow are constructed. The global solutions in \cite{GGT,GNT} are $m$-equivariant with $|\h{0.5pt}m \h{0.5pt}| \geq 3$ and do not blow up. Now we compare our current work with  \cite{AH} in which the solution obtained for the pure heat flow of harmonic maps  also blows up at $t = \infty$. But the blow-up result in \cite{AH} is due to a boundary condition on angular function of orientation variables. A suitably constructed barrier function is utilized in order to  prevent the occurrence of bubbles at finite time. In our Theorem 1.1, the decay rate for the parameter function $\sigma(\cdot)$ is given in \eqref{eq-sigma-exponential}. It is the non-zero twist rate $\mu$ that makes our solution blow up at $t = \infty$. The mechanism for our blow up in Theorem 1.1 is quite different from the work \cite{AH}. As for the other dynamical systems and some finite time blow up results, we refer readers to \cite{GKT, GNT, HLLW, RS} and references therein.\\
 \\
\textbf{I.4. ORGANIZATION OF THE ARTICLE}\\
\\
The article is organized as follows: in Sect.II, we derive equations satisfied by the tangent vector $q$ and the perturbation functions $z_1$, $z_2$ (see \eqref{eq-decom-varphi}). With these equations, in Sect.III, we discuss some fundamental energy estimates and estimates on the modulation parameters $\big(\sigma(t), \Theta(t)\big)$. The proof of Theorem 1.1 and Corollary 1.2 are given in Sect.IV with a bootstrap argument.
\  
\\
\\
\setcounter{section}{2}
\setcounter{footnote}{2}
\setcounter{equation}{0}
\noindent \textbf{II. EQUATIONS OF VARIABLES.}\\
 
In this section we derive some equations that will be used in the study  of \eqref{eq-main-phi-W-V}. 
Given   $\big(W, V, \varphi \big)$  a solution of \eqref{eq-main-phi-W-V}, the vector field $\varphi$ induces a covariant derivative \begin{eqnarray*} \mathrm{D}^{ \varphi}_r  \h{1pt}\mathbf{ e}=\mathbf{ e}_r +\left\langle\mathbf{ e},\, \varphi_{r} \right\rangle  \varphi, \h{20pt}\text{ for all   $\mathbf{e} \in T_{\varphi}\mathbb{S}^2$}. \end{eqnarray*}   Suppose that  $\mathbf{ e}$ is the unique solution of the boundary value problem:
 $$\mathrm{D}^{ \varphi}_r \h{1pt} \mathbf{e}=0,\qquad \mathbf{ e}\left\vert_{r=\infty} =\mathbf{ e_2} .
 \right.$$ Then  $\big\{\mathbf{ e},  \,\varphi\times \mathbf{ e} \big\}$ forms an orthonormal frame on  $T_\varphi\h{1pt} \bS^2$. Therefore for some coefficient  functions $q_{1}$ and $q_{ 2}$, it holds  
 \begin{equation}\label{eq-def-q}
 q\h{0.5pt}[\h{0.5pt}\varphi\h{0.5pt}] =  \p_r \varphi -\frac{m}{r}  \varphi \times \rR \varphi = q_{1} \mathbf{e} + q_{2} \, \varphi \times\mathbf{ e}. 
 \end{equation}Here and in what follows, we assume $m \geq 3$ and denote by $q$ the complexified function $q_{1} + i \h{1pt}q_{2}$. Utilizing similar arguments as in \cite{CSU,CY2,GGT}, by \eqref{eq-main-phi-W-V} and  \eqref{eq-def-q}, we have the following equations satisfied by the unknown variables   $q$, $V$ and $W$:
    \begin{eqnarray}\label{eq-main-q-W-V}
 \left\{
 \begin{aligned}
  & \p_t \h{0.5pt} q  + i \h{0.5pt} S  \h{0.5pt}q  +  i \h{0.5pt}\rL_{m }^* \left[ \left( \frac{m W }{r^2}+ \mu  V  \right)v \right] = - \rL_{m }^*  \rL_{m } q  - \mu^2 \h{1pt}\mathrm{L}_{m}^* \h{0.5pt}\big[ \h{0.5pt}\varphi_{3} \h{0.5pt}v\h{0.5pt} \big];\\[2pt]
  & \rL_{m}^* v =  \varphi_{3} \h{1pt} q; \\[2pt]
 &  \p_t W  =   \p_{rr} W - \frac{1}{r}\p_r  W + m\left(\p_r +\frac{1}{r}\right) \Big\langle q,\,i\h{0.5pt}v\Big\rangle;\\[2pt]
 &   \p_t V    = \Delta_2 V  +  \mu    \left(\p_r +\frac{1}{r}\right) \Big\langle q ,\,i \h{0.5pt}v \Big\rangle.
  \end{aligned}\right.
   \end{eqnarray}
Here   $\rL_{m}$ and its adjoint operator  $\rL_{m}^*$  on $\rL^2$  are given by 
  \begin{equation}\label{eq-Lm-Lm*}
  \rL_{m} = \p_r+\frac{1}{r}-\frac{m\varphi_{3}}{r}\quad\text{and }\quad \rL_{m}^*=-\p_r -\frac{m \varphi_{3}}{r},
  \end{equation} respectively. $v$ in (2.2) is a complex function  $v:= v_{1}+i \h{1pt} v_{2}$ with $v_{1}$ and $v_{2}$ defined by    
  \begin{equation}\label{eq-def-v}
   v_{1}=  -\big\langle\rR\varphi ,\varphi\times {\bf e}\big\rangle = \left\langle {\bf e}_3, {\bf e
  }\right\rangle\quad\mbox{and}\quad v_ {2} =  \big\langle\rR\varphi,   {\bf e}\big\rangle = \left\langle {\bf e}_3, \varphi\times {\bf e}\right\rangle,\end{equation}
   respectively. Moreover $S$ in  \eqref{eq-main-q-W-V} can be read as follows:
 \begin{equation*} 
S = \int_r^\infty \left\langle \rL_{m} q +   \mu^2  \h{1pt} v  \h{1pt} \varphi_{3} +  i \h{0.5pt} v \left(  \frac{m W}{\tau^2} + \mu  \h{0.5pt} V   \right),\, i \h{0.5pt}\Big( q +\frac{m v }{\tau}  \Big) \right\rangle \h{2pt}\mathrm{d} \tau.
 \end{equation*}
 Without ambiguity, we also use $\big<\cdot, \cdot \big>$ to denote the standard inner product on the complex field $\mathbb{C}$.

Letting $\Theta$ and $\sigma$ be two   time dependent modulation parameters, we suppose that  $\varphi$ admits  a decomposition  as shown in \eqref{eq-decom-varphi}. Here for $z = z_1 + i \h{0.5pt}z_2$ satisfying $ \| \h{0.5pt}z\h{0.5pt}\|_{\mathrm{L}^{\infty}} \leq 1/2$, $\gamma$ in \eqref{eq-decom-varphi} is given by  
\begin{equation}\label{eq-def-gamma}
\gamma=\left(1-|z|^2\right)^{1/2}-1. 
\end{equation} Moreover it holds 
\begin{eqnarray}
\label{eq-prop-gamma}
\left|\h{1pt}\gamma\h{1pt}\right|\,\lesssim\, \left|\h{1pt}z\h{1pt}\right|^2 \quad\mbox{and}\quad \left|\h{1pt}\p_\rho\gamma\h{1pt}\right| \,\lesssim\, \left|\h{1pt}z\h{1pt}\right|\cdot\left|\h{1pt}\p_\rho\h{1pt} z\h{1pt}\right|.
\end{eqnarray}
Now we plug  \eqref{eq-decom-varphi} into \eqref{eq-main-phi-W-V} and  obtain the following equation satisfied by    $z$:   
 \begin{equation} \label{eq-z-Mod-HT}
   \p_t \h{0.5pt} z  + \frac{1}{\sigma^2}\h{1pt} N  \h{0.5pt} z  = \mathrm{Mod} + \mathrm{HT}, \end{equation}
   where
   \begin{eqnarray}\label{eq-def-Mod-HT}
\begin{aligned}
 & \mathrm{Mod} := - \Big\{ \big( 1 + \gamma \big) \mathbf{h}_{1} + i\h{0.5pt}\mathbf{h}_{3}\h{0.5pt}z \Big\} \left(  \Theta'  + \mu\h{0.5pt}V + \frac{m W}{r^2} \right)  \\[4pt]
 & \h{20pt}+ \frac{\sigma'}{\sigma} \h{1pt}\Big\{ i \h{1pt}\big( 1 + \gamma \big) m \mathbf{h}_{1} + \rho\h{1pt}\p_{\rho}\h{0.5pt}z \Big\} + \mu^2 \Big\{ i \h{1pt}\big( 1 + \gamma \big) \mathbf{h}_{1}\mathbf{h}_{3} + i \mathbf{h}_{1}^2 z_{2} - \mathbf{h}_{3}^2 \h{1pt} z \Big\}; \\[10pt] 
 &\mathrm{HT} := \frac{i}{\sigma^2} \h{1pt}\frac{2m\mathbf{h}_{1}}{\rho}\h{1pt}\p_{\rho} \gamma \\[4pt]
 & \h{18pt} +  \left( \frac{m^2}{\rho^2 \sigma^2} + \mu^2\right) z  \Big\{ z_{1}^2 + \big( \gamma \mathbf{h}_{1} - z_{2} \mathbf{h}_{3} \big)^2 + 2 \mathbf{h}_{1} \big( \gamma  \mathbf{h}_{ 1} - z_{ 2} \mathbf{h}_{ 3} \big) \h{1pt}\Big\} \\[8pt]
& \h{18pt} + \frac{1}{\sigma^2} \left\{   \left( \p_{\rho} \gamma - \frac{m \mathbf{h}_{ 1}}{\rho} z_{ 2} \right)^2 + \big( \p_{\rho} \h{1pt}z_{ 1}\big)^2 + \left( \p_{\rho}\h{0.5pt}z_{  2} + \frac{m \mathbf{h}_{ 1}}{\rho} \gamma  \right)^2 + \frac{2m\mathbf{h}_{1}}{\rho} \left( \h{1pt} \p_{\rho}\h{0.5pt}z_{2} + \frac{m \mathbf{h}_{1}}{\rho} \gamma \right) \h{1pt}\right\} z.
 \end{aligned}
 \end{eqnarray}
In \eqref{eq-def-Mod-HT}, $\Theta' $ and $ \sigma'$ represent  the time derivatives of $ \Theta $ and $ \sigma$, respectively. The operator $N $ is defined by
\begin{equation}\label{eq-def-N}
-N  \,:=\,-\rL_{\mathbf{ h} }^* \h{1pt}\rL_{\mathbf{h} }
\,=\,\p_{\rho\rho}+\frac{1}{\rho}\h{0.5pt}\p_\rho+\frac{m^2}{\rho^2}\Big(2 \mathbf{h}_{1} ^2-1\Big), 
\h{5pt}\text{where}\h{4pt} \rL_{\bh}\,=\,\p_\rho+\frac{m}{\rho}\h{1pt}\mathbf{h}_{3}(\rho) .
\end{equation}
Here   $\rL_{\mathbf{ h} }^*$ is the adjoint operator of $\mathrm{L}_{\mathbf{h}}$ on  $\rL^2\left(\rho \h{1pt} \mathrm d\h{0.5pt}\rho\right)$.  

Before proceeding we give some preliminary results for later use. For the variable $z$,   we equip it with the norm in  the space $X$ (see  \eqref{Xnorm}). By Sobolev embedding,  we have 
\begin{equation}\label{eq-z-sobolev}
\mbox{$z$ is continuous on $[\h{1pt}0, \infty )$ with $z(0)=z(\infty)=0\,$ and satisfies $\displaystyle \left\|\h{0.5pt}z\h{0.5pt}\right\|_{\mathrm L^\infty}\lesssim \left\|\h{0.5pt} z\h{0.5pt} \right\|_X$ }.
\end{equation} In light of \eqref{eq-def-v} and  \eqref{eq-decom-varphi},    it follows
 \begin{equation}\label{eq-v-h_1-z}
 \left|\h{0.5pt}v\h{0.5pt}\right| ^2 \h{1pt}=\h{1pt}\left|\h{1pt}\rR \varphi\h{1pt}\right|^2\h{2pt}\lesssim\h{2pt}\mathbf  h_1^2(\rho)\h{1pt}+\h{1pt} \left|\h{1pt}z(\rho,t)\h{1pt}\right|^2, \h{20pt}\text{provided that $\|\h{1pt}z\h{1pt}\|_{\mathrm{L}^{\infty}} \leq 1/2$.}
 \end{equation}Now we consider the operator $\mathrm{L}_{\mathbf{h}}$ in \eqref{eq-def-N}. Its kernel space is non-trivial and satisfies   $ \Ker \mathrm L_{\mathbf h} \h{1pt}= \h{1pt} \Span\big\{\h{0.5pt}\mathbf h_1\h{0.5pt}\big\}$, where the function  $\mathbf h_1$ is defined in   \eqref{equiv-harmonic}. Moreover $\mathrm{L}_{\mathbf{h}}$ satisfies the following coercivity result: 

  \begin{lemma}[Lemma 2.4 in \cite{ GKT}]\label{lem-Coercivity}  
  Let $m$ be an integer with  $|m|\geq 3 $.  If  $z\in   X$ satisfies
\begin{equation}\label{eq-orthogonal-z-h_1-no-time}
\int_0^\infty \h{1.5pt} z  \h{1pt}     \h{1pt}  \mathbf h_1 \h{1.5pt} \rho\h{0.5pt} \mathrm d\h{0.5pt} \rho  \h{1pt} =\h{1pt} 0,
   \end{equation}  then we have
\begin{equation}\label{eq-coer-L_h}
  \left \|\h{0.5pt} z\h{0.5pt} \right\|_{X}^2\h{2pt}  \lesssim  \h{2pt}  \int_0^{\infty} \big|\h{1pt}\mathrm{L}_{\mathbf{h}} \h{1pt}z\h{1pt}\big|^2 \rho\h{1pt} \mathrm d \h{1pt}\rho  .
 \end{equation}  
   \end{lemma}\noindent Moreover  we have the following equivalence of $\| \h{1pt}z \h{1pt}\|_X$ and $\|\h{1pt}q\h{1pt}\|_{\mathrm{L}^2}$:
 \begin{lemma}\label{equiv-norms} Suppose that $z \in X$ and  satisfies the orthogonality condition \eqref{eq-orthogonal-z-h_1-no-time}. $q$ is the complexified function $q_1 + i\h{1pt}q_2$, where $q_1$ and $q_2$ are defined in \eqref{eq-def-q}.  There exists a positive constant $\e_0 = \e_0(m)$ suitably small so that 
  if  $\left\|\h{1pt}z\h{1pt}\right\|_X \h{1pt}<\h{1pt}\e_0$,  then    it holds    
  \begin{equation*}\label{eq-est-z-q}
   \left\|\h{1pt} q\h{1pt} \right\|_{\rL^2} \h{2pt}\lesssim \h{2pt} \left \|\h{1pt} z\h{1pt} \right\|_{X}\h{2pt} \lesssim  \h{2pt}  \left\|\h{1pt} q\h{1pt} \right\|_{\rL^2} . 
 \end{equation*}   \end{lemma}  \noindent  The first inequality in Lemma 2.4 can be obtained by \eqref{eq-decom-varphi} and \eqref{eq-def-q}. The second inequality in Lemma 2.4 is due to  Proposition 2.3 in \cite{GKT}. To end this section,  we state a coercivity result  on the operator   $\mathrm L_{m}$ (see \eqref{eq-Lm-Lm*}). We refer to Lemma 4.1 in \cite{CY2} for  the proof. 
  \begin{lemma} There exists a positive constant $\e_0 = \e_0(m)$ suitably small so that 
  if  $\left\|\h{1pt}z\h{1pt}\right\|_X \h{1pt}<\h{1pt}\e_0$,  then    it holds   
\begin{equation*}\label{eq-coer-L_m}
 \int_0^{\infty} \, \left| \h{1pt}\p_r  \h{0.5pt}q \h{1pt}\right|^2  \h{0.5pt}+ \h{0.5pt}\frac{\left| \h{1pt}q \h{1pt}\right|^2}{r^2}\,  
\,\,   \lesssim   \,\,   \int_0^{\infty} \,\left| \h{1pt} \mathrm{L}_{m} \h{1pt}q\h{1pt}\right|^2.
   \end{equation*}
   \end{lemma}
   \
  \\
  \\
\setcounter{section}{3}
\setcounter{thm}{0}
\setcounter{equation}{0}
\noindent \textbf{III.  ESTIMATES OF ENERGY AND MODULATION PARAMETERS.}\\

In this section we study various energy estimates related to a solution $\big(\h{0.5pt}W, V, q \h{0.5pt}\big)$ of  \eqref{eq-main-q-W-V}. \\
\\
\textbf{III.1. FUNDAMENTAL ENERGY ESTIMATES.} 
 \\
 
\noindent Recalling the Oseen part $W^{os}$ in \eqref{eq-def-Oseen}, we have$$ \p_t W^{os}   = \p_{rr} W^{os} - r^{-1} \p_r W^{os} .$$
 Therefore by the third equation in \eqref{eq-main-q-W-V}, $W^* := W - W^{os}$ satisfies 
 \begin{equation}\label{eq-W*}
  \p_t W^*  =   \p_{rr} W^* - \frac{1}{r}\h{0.5pt}\p_r  W^* + m\left(\p_r +\frac{1}{r}\right) \Big\langle q,\,i \h{0.5pt}v\Big\rangle.
\end{equation}
In the following lemma, we give an energy identity related to $q $, $v $, $V $ and $W^*$. 

\begin{lemma}\label{lem-energy-identity} 
 Suppose that $q$, $v$, $V$ and $W^*$ satisfy \eqref{eq-main-q-W-V} and \eqref{eq-W*}. Then it holds  
  \begin{equation}\label{eq-energy-identity} 
  \frac{1}{2} \frac{\mathrm{d}}{\mathrm{d} \h{0.5pt}t} \h{1pt} \mathrm{E} + \int_0^{\infty}   \left| \h{1pt}\mathrm{L}_{m} q  + \mu^2  \h{1pt}v \h{1pt}\varphi_{3}  \h{1pt}\right|^2 +  \big( \p_r V \big)^2 + \frac{ \big( \p_r W^* \big)^2}{r^2}  =  m \int_0^{\infty} \left( \frac{W^{os}}{r^2} \right)_r \big< \h{0.5pt}q, i \h{0.5pt}v \h{0.5pt}\big>,
\end{equation}
where 
\begin{equation*} 
\mathrm{E} :=  \mathrm{E}^{*} +  \mu^2  \int_0^{\infty}  | \h{1pt}v \h{1pt}|^2 \h{15pt}\text{with}\h{8pt} \mathrm{E}^* :=  \int_0^{\infty}   |\h{1pt}q\h{1pt}|^2 + V^2 + \frac{\big( W^* \big)^2}{r^2} .
\end{equation*}
\end{lemma}

\begin{proof}[\bf Proof]
Multiplying $W^* \big/ r$ on both sides of \eqref{eq-W*} and integrating over $(0, \infty)$, we obtain 
\begin{equation}\label{eq-energy-est-W*} 
\frac{1}{2}\h{1pt}\frac{\mathrm{d}}{\mathrm{d} t} \int_0^{\infty} \frac{\big( W^* \big)^2}{r^2} + \int_0^{\infty} \frac{ \big( \p_r W^* \big)^2}{r^2} + m \int_0^{\infty} \left( \frac{W^*}{r^2} \right)_r \big< \h{0.5pt}q, i \h{0.5pt}v \h{0.5pt}\big> = 0.
\end{equation} 
Multiplying $r V$ on both sides of the last equation in \eqref{eq-main-q-W-V} and integrating over $(0, \infty)$, we get 
\begin{equation} \label{eq-energy-est-V} 
\frac{1}{2}\h{1pt}\frac{\mathrm{d}}{\mathrm{d} t} \int_0^{\infty} V^2 + \int_0^{\infty}  \big( \p_r V \big)^2  + \mu \int_0^{\infty}  \p_r V  \h{1.5pt}  \big< \h{0.5pt}q, i \h{0.5pt}v \h{0.5pt}\big> = 0.
\end{equation}
By taking inner product with $r q $ on both sides of the first equation in \eqref{eq-main-q-W-V} and integrating over $(0, \infty)$, it follows 
\begin{equation}\label{eq-energy-est-q-0}  
 \frac{1}{2}\h{1pt}\frac{\mathrm{d}}{\mathrm{d} t} \int_0^{\infty} |\h{1pt}q\h{1pt}|^2 + \int_0^{\infty} \big| \h{1pt}\mathrm{L}_{m} q \h{1pt}\big|^2 +     \int_0^{\infty} \left< q, i \h{1pt}\rL_{m}^* \left[ \left(\frac{m W}{r^2}+ \mu V  \right)v\right] \h{1pt} \right>  = -  \mu^2 \int_0^{\infty} \big< \h{1pt}\mathrm{L}_{m} q,   \h{0.5pt}\varphi_{3} \h{0.5pt}v \h{1pt}\big>.
\end{equation}
Using the definition of $\mathrm{L}_m^*$ in \eqref{eq-Lm-Lm*} and the second equation in \eqref{eq-main-q-W-V}, one can show that 
\begin{equation}\label{eq-integration-by-part-Lm}
\begin{aligned}
 i\h{1pt}\rL_{m}^* \left[ \left(\frac{m W}{r^2}+ \mu  V   \right)v\right]  &= -  i\h{1pt} v \left(\frac{m W}{r^2}+  \mu V   \right)_r +  \left(\frac{m W}{r^2}+ \mu V  \right)  i\h{1pt} \mathrm{L}_{m}^* v  \\[4pt]
&= - i\h{1pt}  v \left(\frac{m W}{r^2}+ \mu  V   \right)_r +  \left(\frac{m W}{r^2}+ \mu V  \right)  i\h{1pt} \varphi_{3} \h{1pt} q.  
\end{aligned}
 \end{equation}
Applying this equality to \eqref{eq-energy-est-q-0} yields 
\begin{equation}\label{eq-energy-est-q-1}
  \frac{1}{2}\h{1pt}\frac{\mathrm{d}}{\mathrm{d} t} \int_0^{\infty} |\h{1pt}q\h{1pt}|^2 + \int_0^{\infty} \big| \h{1pt}\mathrm{L}_{m} q \h{1pt}\big|^2 -     \int_0^{\infty} \left(\frac{m W}{r^2}+ \mu V   \right)_r \left<\h{0.5pt} q, i \h{1pt}v \h{1pt} \right>  = - \mu^2 \int_0^{\infty} \big< \h{1pt}\mathrm{L}_{m} q, \varphi_{3}v \h{0.5pt}   \h{1pt}\big>.
\end{equation}
This combined with \eqref{eq-energy-est-W*}-\eqref{eq-energy-est-V} implies the following identity  
 \begin{equation}\label{eq-energy-est-q-2}
 \frac{1}{2}\h{1pt}\frac{\mathrm{d}}{\mathrm{d} \h{1pt}t} \h{1.5pt} \mathrm{E}^* + \int_0^{\infty}  \big| \h{1pt}\mathrm{L}_{m} q \h{1pt}\big|^2 + \big( \p_r V \big)^2 + \frac{ \big( \p_r W^* \big)^2}{r^2} + \mu^2 \int_0^{\infty} \big< \h{1pt}\mathrm{L}_{m} q, \varphi_{3} \h{0.5pt}v \h{1pt}\big> =   \int_0^{\infty} \left( \frac{m W^{os}}{r^2} \right)_r \big< \h{0.5pt}q, i \h{0.5pt}v \h{0.5pt}\big>.
\end{equation}
Taking inner product with $r \h{1pt}(\varphi_{1}, \varphi_{2}, 0)^t$ on both sides of the first equation in \eqref{eq-main-phi-W-V}
 and integrating over $(0, \infty)$, we get 
 \begin{equation}\label{eq-energy-est-v-0}
 \begin{aligned}
  \frac{1}{2} \h{1pt}\frac{\mathrm{d}}{\mathrm{d} \h{1pt}t } \int_0^{\infty} \big| \h{1pt}v \h{1pt}\big|^2  + \mu^2 \int_0^{\infty} \big|\h{1pt}v\h{1pt}\big|^2 \varphi_{3}^2 
   & =   \int_0^{\infty} \Big< \h{1pt} \big( \mathrm{L}_{m} q_{1} \big) \mathbf{e} + \big( \mathrm{L}_{m} q_{2} \big) \varphi \times \mathbf{e}\,,\, \sum_{j = 1}^2 \varphi_{ j} \h{1pt} \mathbf{e}_j\,\Big> \\[2pt]
 &= - \int_0^{\infty} \big( \mathrm{L}_{m } q_{ 1} \big) \varphi_{  3} \big< \h{1pt}\mathbf{e}_3, \mathbf{e}  \h{1pt}\big> + \big( \mathrm{L}_{m} q_{2} \big) \varphi_{3} \big< \h{1pt}\mathbf{e}_3, \varphi
 \times \mathbf{e} \h{1pt}\big>.\\[2pt]
 \end{aligned}
\end{equation}
Here  in the first equality above we   have used the fact that 
 $$  \Delta_2\h{0.5pt}\varphi  + \big\vert \h{0.5pt} \p_{r} \varphi \h{1pt}\big\vert^2 \varphi    +  \frac{m^2}{r^2}   \bigg[\rR^2  \varphi+ \big\vert \h{1.5pt} \rR \varphi \h{0.5pt}\big\vert^2 \varphi\bigg] = \big( \mathrm{L}_{m} q_{1} \big) \h{1pt}\mathbf{e} + \big( \mathrm{L}_{m} q_{2} \big) \h{1pt}\varphi \times \mathbf{e},$$ while  for the second equality in \eqref{eq-energy-est-v-0},  the orthogonality of $\varphi$  with $\mathbf{e} $ and $\varphi \times \mathbf{e}$ is applied.
  Using the definition of $v$ in \eqref{eq-def-v}, 
  we can rewrite the equality \eqref{eq-energy-est-v-0}  as follows: 
  \begin{eqnarray*} \frac{1}{2} \h{1pt}\frac{\mathrm{d}}{\mathrm{d} \h{0.5pt}t} \int_0^{\infty} |\h{1pt}v \h{1pt}|^2  + \mu^2 \int_0^{\infty}  |\h{1pt}v \h{1pt}|^2  \varphi_{3}^2  =  - \int_0^{\infty} \big<\h{1pt} \mathrm{L}_{m} q, \varphi_{3}\h{1pt}v \h{1pt}\big>.
\end{eqnarray*}
Multiplying   both sides above by  $\mu^2$ and summing with \eqref{eq-energy-est-q-2}, we deduce \eqref{eq-energy-identity}. 
\end{proof}

By using \eqref{eq-energy-identity}, the energy $\mathrm{E}$ can be uniformly bounded for all $  t>0$. That is 
 \begin{lemma}\label{lem-energy-est-uniform-bound}
Suppose that  $q$, $v$, $V$ and $W^*$ satisfy \eqref{eq-main-q-W-V} and \eqref{eq-W*}.  Then we have, for all $ t>0$, that
 \begin{eqnarray}\label{eq-energy-est-uniform-bound}
  \mathrm{E} (t) +   \int_0^t \int_0^{\infty}   \left| \h{1pt}\mathrm{L}_{m } q   + \mu^2\h{1pt} v  \h{1pt}\varphi_{3}  \h{1pt}\right|^2  +  \big( \p_r V  \big)^2 + \frac{ \big( \p_r W^* \big)^2}{r^2}   \h{2pt}  \lesssim \h{2pt} \mathrm{E}(0).
\end{eqnarray} 
\end{lemma}\begin{proof}[\bf Proof]
Using   integration by parts  and the second equation of (2.2), we have
 \begin{equation}\label{eq-est-W-os-Lm-q} 
 \begin{aligned}
 m \int_0^{\infty} \left( \frac{W^{os}}{r^2} \right)_r \big< \h{0.5pt}q, i \h{0.5pt}v \h{0.5pt}\big>& = \h{2pt} - m \int_0^{\infty}  \frac{W^{os}}{r^2} \big< \h{0.5pt} \mathrm{L}_{m} q, i \h{0.5pt}v \h{0.5pt}\big> =  - m \int_0^{\infty}  \frac{W^{os}}{r^2} \left< \h{0.5pt} \mathrm{L}_{m} q  + \mu^2 \h{0.5pt}  v \h{1pt}\varphi_{3} , i \h{0.5pt}v \h{0.5pt}\right> \\[4pt]
&\leq \h{2pt} \frac{1}{2} \int_0^{\infty} \left| \h{1pt} \mathrm{L}_{m} q  + \mu^2\h{1pt} v \h{1pt}\varphi_{3}  \h{1pt}\right|^2 + \frac{m^2}{2} \int_0^{\infty} |\h{1pt}v\h{1pt}|^2 \left( \frac{W^{os}}{r^2} \right)^2.
\end{aligned}
\end{equation}
Since it holds from   the definition of $W^{os}$ in  \eqref{eq-def-Oseen} that     \begin{equation}\label{eq-W^os-bound}
 \big|\h{1.5pt}W^{os}\h{1pt}\big| \,\leq \,|\h{1pt}\omega\h{1pt}|\h{2pt}  r^2\h{1pt} \big/ \h{2pt} l \qquad \hbox{ 
  on $\mathbb{R}^+$,}
  \end{equation}
 the estimate in \eqref{eq-est-W-os-Lm-q} then implies 
\begin{equation*} m \int_0^{\infty} \left( \frac{W^{os}}{r^2} \right)_r \big< \h{0.5pt}q, i \h{0.5pt}v \h{0.5pt}\big>\, \leq\,  \frac{1}{2} \int_0^{\infty} \left| \h{1pt} \mathrm{L}_{m} q  +  \mu^2\h{1pt}v \h{1pt}\varphi_{3}  \h{1pt}\right|^2 + \frac{m^2 \h{0.5pt} \omega^2}{2} \h{1pt} l^{-2} \int_0^{\infty} |\h{1pt}v\h{1pt}|^2. 
\end{equation*}
  Applying this estimate to the right-hand side of \eqref{eq-energy-identity}, we obtain 
  \begin{equation}\label{eq-energy-E-est} 
   \frac{\mathrm{d}}{\mathrm{d} t} \h{1pt} \mathrm{E} + \int_0^{\infty}    \left| \h{1pt}\mathrm{L}_{m} q  + \mu^2 \h{1pt}v \h{1pt}\varphi_{3}  \h{1pt}\right|^2 +  \big( \p_r V \big)^2   + \frac{ \big( \p_r W^* \big)^2}{r^2}   \h{2pt}  \leq \h{2pt}  \frac{ m^2\h{0.5pt}\omega^2 }{\mu^2 \h{0.5pt}l^2}   \h{2pt}\mathrm{E}.
\end{equation}Solving this ODE inequality for $\mathrm{E}$ gives us 
\begin{eqnarray*}\mathrm{E}(t)\,\leq\,
\mathrm{E}(0) \h{1pt}\exp\left\{ \h{1pt}\frac{ m^2\h{1pt}\omega^2 }{\mu^2 \h{1pt}r_0^2}\h{2pt} \frac{t}{4t+r_0^2} \h{1pt} \right\}, \end{eqnarray*}
which can be applied to the right-hand side of  \eqref{eq-energy-E-est} and yields the energy inequality  \eqref{eq-energy-est-uniform-bound}.
\end{proof}
As a corollary it follows
\begin{cor}\label{cor-energy-est-uniform-bound} 
 With the same assumptions as in Lemma \ref{lem-energy-est-uniform-bound}, 
for any positive constant $\e_*$, there exists a  small positive constant $\delta_*$ such that if \eqref{eqq-assump-small-initial} holds, then for all $t>0$, we have 
\begin{equation}\label{energy-est-uniform-bound-e*}
  \,\mathrm{E} (t) +   \int_0^t \int_0^{\infty}   \left| \h{1pt}\mathrm{L}_{m } q   + \mu^2\h{1pt} v  \h{1pt}\varphi_{3}  \h{1pt}\right|^2  +  \big( \p_r V  \big)^2 + \frac{ \big( \p_r W^* \big)^2}{r^2}   \h{2pt}   <\, \e_*^2.
\end{equation} Here $\delta_*$ depends on $m,$  $ \mu,$ $\omega,$  $r_0$ and  $\e_*.$
\end{cor}
\begin{proof}[\bf Proof]
In light of the definition of $\mathrm{E}$ in Lemma 3.1 and the energy estimate in Lemma 3.2, we only need to bound the $\mathrm{L}^2$-norms of $v_{in}$ and $q_{in}$. Here $v_{in}$ and $q_{in}$ are initial functions of $v$ and $q$ at $t = 0$. 
By  \eqref{eq-v-h_1-z}, it holds 
\begin{equation}\label{eq-est-L2-v-h_1-z}
  \int_0^{\infty}  | \h{1pt}v_{in}  \h{1pt}|^2\, \lesssim\,    \sigma_{in}^2  \int_0^{\infty}   \big( \h{1pt}\mathbf h_1^2 + |\h{1pt}z_{in} \h{1pt}|^2 \h{1pt}\big) \rho\h{1pt}\mathrm{d}\h{1pt}\rho\,\lesssim \, \sigma_{in}^2  + \sigma_{in}^2  \int_0^{\infty}  |\h{1pt}z_{in}\h{1pt}|^2 \rho \h{1pt}\mathrm{d}\h{1pt}\rho.
\end{equation}
Applying Lemma 2.4 then yields
 \begin{equation*} \|\h{1pt}q_{in}\h{1pt}\|_{\mathrm{L}^2} \h{3pt}\lesssim\h{3pt}\|\h{1pt}z_{in}\h{1pt}\|_X.
\end{equation*}
Therefore by \eqref{eqq-assump-small-initial}, the $\mathrm{L}^2$-norms of $v_{in}$ and $q_{in}$ can be small, provided that $\delta_*$ is suitably small. The proof then follows.
\end{proof} \
\\
\\
 \noindent \textbf{III.2. DECOMPOSITION OF THE VARIABLES $V$ AND $W^*$.}\\
 
In the remainder of this article, we decompose $V $ into $V_{1} + V_{2}$, where $V_{1}$ and $V_{2}$ satisfy respectively the following   initial value problems: 
\begin{equation}\label{eq-for-V_1}
\left\{
\begin{aligned}
 & \p_t V_{1}   = \La_2 V_{1}  \quad&&\text{in $\R^2\times (0,\infty)$};\\[6pt]
  &\h{8pt} V_{1} =  V_{in}   &&\text{at $t=0$;}
\end{aligned}
\right.\end{equation}
and 
\begin{equation}\label{eq-for-V_2} 
\left\{
\begin{aligned}
 &  \p_t V_{2}   = \La_2 V_{2} + \mu \left( \p_r +  {r^{-1}} \right) \big< \h{1pt}q, i \h{0.5pt}v\h{1pt}\big>\qquad&&\mbox{in $\R^2\times(0,\infty)$};\\[6pt]
  & \h{8pt}  V_{2} =  0   &&\mbox{at $t=0$.}
\end{aligned}
\right. 
\end{equation}
A standard application of the heat kernel on $\mathbb{R}^2$ implies the following lemma.  \begin{lemma}\label{lem-est-V_1}
For all $t > 0$, it holds \begin{eqnarray*}
\Big\| \h{1pt} V_{1}\big(\h{0.5pt}  \cdot ,t\h{0.5pt}\big) \h{1pt}\Big\|_{\mathrm{L}^{\infty}}    \h{2pt}\lesssim\h{2pt} t^{- 1 / 2 } \h{1pt}\big\| \h{1pt}V_{in} \h{1pt} \big\|_{\mathrm{L}^2}.
\end{eqnarray*}
\end{lemma}

Now we consider some estimates of $V_2$.  In what follows we use $\mathscr{F}_n$ to denote the standard Fourier transformation on $\mathbb{R}^n$. Firstly we estimate the Fourier transform of $V_2$ on $\mathbb{R}^2$.   
\begin{lemma}\label{lem-est-V_2-fourier}
 For all $t  >0$ and $\xi\in\R^2$,   it holds 
  $$\Big\vert \h{1.5pt}\mathscr{F}_2\h{1pt}[\h{1pt}{V_{2}}\h{1pt}]\big(\xi \h{1pt}\big) \h{1pt}\Big\vert \h{2pt}\lesssim\h{2pt}   |\hspace{1pt}\xi \hspace{1.5pt}| \int_0^t \| \hspace{1pt}q \hspace{1pt}\|_{\mathrm{L}^2}\hspace{1pt}\|\hspace{1pt}v\hspace{1pt}\|_{\mathrm{L}^2}  \h{1pt}\mathrm{d} s.$$ 
\end{lemma}
\begin{proof}[\bf Proof] 
Taking Fourier transform $\mathscr{F}_2$ on both sides of the equation \eqref{eq-for-V_2}, we obtain for all $\xi \in \mathbb{R}^2$ that 
\begin{eqnarray*}
  \p_t \hspace{0.5pt}\sF_2 \left[ \h{1pt}V_{2} \h{1pt}\right]+ |\hspace{1pt}\xi\hspace{1pt}|^2 \hspace{1pt}\sF_2 \left[ \h{1pt}V_{2} \h{1pt}\right]= \mu\h{1pt}\mathscr{F}_2 \left[  \hspace{1pt} \left( \p_r + r^{-1} \right) \big< \hspace{0.5pt}q, i\h{0.5pt}v  \hspace{0.5pt}\big> \hspace{1pt}\right].
\end{eqnarray*}
Letting $\nabla^{(2)} $ be the gradient operator on $\mathbb{R}^2$, then we  have
\begin{eqnarray*}
\big( \hspace{1pt}\p_r+r^{-1} \hspace{1pt}\big) \hspace{1pt}\big<\hspace{0.5pt} q , i\h{0.5pt}v \hspace{0.5pt}\big> = \nabla^{(2)} \cdot \big(\hspace{1pt} x \hspace{1pt}r^{-1}\big<\hspace{0.5pt} q , i\h{0.5pt}v \hspace{0.5pt}\big>\hspace{1pt}\big) \quad\hbox{for $x\in\R^2$}.
\end{eqnarray*} 
The above two equalities together with  H\"older's inequality imply   
\begin{eqnarray*}
&&\Big|\hspace{1pt}\sF_2 \left[ \h{1pt}V_2 \h{1pt}\right] \left(\xi\right) \hspace{1pt}\Big|=  \left|  \hspace{2pt}\mu\,\int_0^t e^{-(t-s) \hspace{1pt}|\hspace{1pt}\xi\hspace{1pt}|^2} \sum_{j = 1}^2  \xi_j \hspace{1pt}\mathscr{F}_2\left[ x_j\hspace{1pt}r^{-1} \hspace{1pt}\big<\hspace{0.5pt} q , i\h{0.5pt}v \hspace{0.5pt}\big> \hspace{1pt}\right]\,\mathrm{d} s   \hspace{2pt} \right|  \\[6pt] 
&&\hspace{20pt}\lesssim\hspace{2pt}  |\hspace{1pt}\xi \hspace{1.5pt}| \int_0^t  \big\| \hspace{1.5pt} x \hspace{1pt}r^{-1} \hspace{1pt}\big< \hspace{0.5pt}q , i\h{0.5pt} v  \hspace{0.5pt}\big> \hspace{1.5pt}\big\|_{\mathrm{L}^1(\mathbb{R}^2)} \h{2pt}\mathrm{d} s 
\hspace{2pt} \lesssim\h{2pt} |\hspace{1pt}\xi \hspace{1.5pt}| \int_0^t \int_0^{\infty} | \hspace{1pt}q  \hspace{1pt}|\hspace{1pt}|\hspace{1pt}v \hspace{1pt}| \hspace{1pt}r\hspace{2pt}\mathrm{d} r \h{1pt}\mathrm{d} s 
\hspace{2pt}\lesssim\hspace{2pt}   |\hspace{1pt}\xi \hspace{1.5pt}| \int_0^t \| \hspace{1pt}q  \hspace{1pt}\|_{\mathrm{L}^2}\hspace{1pt}\|\hspace{1pt}v\hspace{1pt}\|_{\mathrm{L}^2}  \h{1pt}\mathrm{d} s.
\end{eqnarray*}
This finishes the proof.  
\end{proof}

Now we give an energy estimate  for $V_2$ based on Schonbek's Fourier splitting method (see \cite{S}). 
\begin{lemma} \label{lem-energy-est-V_2}
 For all $t  >0$ and $\mathrm{R}_* > 0$,    it holds  
 \begin{eqnarray*} 
 \frac{\mathrm{d}}{\mathrm{d} \h{0.5pt}t} \int_0^{\infty }V_2^2 + \mathrm{R}_*^2 \int_0^{\infty} V_2^2  \h{2pt}
 \lesssim\h{2pt} \|\h{1pt}z \h{1pt}\|_X^2 +   \mathrm{R}_*^6 \left(  \int_0^t \| \hspace{1pt}q  \hspace{1pt}\|_{\mathrm{L}^2}\hspace{1pt}\|\hspace{1pt}v\hspace{1pt}\|_{\mathrm{L}^2}  \h{1pt}\mathrm{d} s \right)^2.
 \end{eqnarray*}
\end{lemma}
\begin{proof}[\bf Proof] 
By Plancherel's  theorem, it turns out
\begin{eqnarray*}
&&  \int_{\R^2} \left|\h{1pt}\D^{(2)}  V_{2}\h{1pt}\right|^2 \mathrm{d} x =  \int_{\R^2} |\h{1pt}\xi\h{1pt}|^2 \h{1pt} \Big|\h{1pt}\mathscr{F}_2 \h{1pt}\big[  V_{2} \big]\h{1.5pt} \Big|^2 \mathrm{d} \h{0.5pt}\xi \\[8pt]
\nonumber&& 
  \h{64pt} \geq \h{2pt}  \mathrm{R}_*^2 \int_{ \R^2 \h{0.5pt}\setminus \h{0.5pt} \mathrm{D}_{\mathrm{R}_*}} \Big|\h{1pt}\mathscr{F}_2 \h{1pt}\big[  V_{2} \big]\h{1.5pt} \Big|^2 \mathrm{d} \h{0.5pt}\xi     =  \mathrm{R}_*^2 \int_{\R^2} V_{2}^2 \h{3pt} \mathrm{d} x  - \mathrm{R}_*^2   \int_{\mathrm{D}_{\mathrm{R}_*}}    \Big|\h{1pt}\mathscr{F}_2 \h{1pt} \big[  V_{2} \big]\h{1.5pt} \Big|^2 \mathrm{d}\h{0.5pt}\xi. 
\end{eqnarray*}
Here $\mathrm{D}_{\mathrm{R}_*}$ is the disk in $\mathbb{R}^2$ with radius $\mathrm{R}_*$ and  center  $0$.  This estimate then yields  
 \begin{equation}\label{eq-est-V_2-plancherel}
 \int_0^{\infty} \big(\h{1pt}\p_r V_{2}\h{1pt}\big)^2 \h{2pt}\geq\h{2pt}  \mathrm{R}_*^2 \int_0^{\infty} V_{2}^2  - \frac{\mathrm{R}_*^2}{2 \h{0.5pt}\pi}   \int_{\mathrm{D}_{\mathrm{R}_*}}    \Big|\h{1pt}\mathscr{F}_2 \h{1pt}\big[  V_{2} \big]\h{1.5pt} \Big|^2 \mathrm{d}\h{0.5pt}\xi 
\end{equation}
since
\begin{eqnarray*}
 \int_{\R^2} \left|\h{1pt}\D^{(2)}  V_{2}\h{1pt}\right|^2 \mathrm{d} x = 2 \h{0.5pt}\pi  \int_0^{\infty} \big(\h{1pt}\p_r V_{2}\h{1pt}\big)^2.
\end{eqnarray*}
Multiplying $V_2$ on both sides of the equation \eqref{eq-for-V_2} and integrating over $ \R^2$, we deduce 
\begin{equation}\label{eq-energy-V_2}
\frac{1}{2} \h{1pt}\frac{\mathrm{d}}{\mathrm{d} \h{0.5pt}t} \int_0^{\infty }V_2^2 + \int_0^{\infty} \big( \p_r \h{0.5pt}V_2 \big)^2 = - \mu \int_0^{\infty} \big< \h{1pt}q, i \h{0.5pt}v \h{1pt}\big> \h{1pt}\p_r \h{0.5pt}V_2,
\end{equation}
which   implies in light of  Young's inequality 
\begin{eqnarray*}
\frac{\mathrm{d}}{\mathrm{d} \h{0.5pt}t} \int_0^{\infty }V_2^2 + \int_0^{\infty} \big( \p_r \h{0.5pt}V_2 \big)^2 \h{2pt}\leq \h{2pt}\mu^2 \int_0^{\infty} \big| \h{1pt}\big< \h{1pt}q, i \h{0.5pt}v \h{1pt}\big> \h{1pt}\big|^2.
\end{eqnarray*}
Applying \eqref{eq-est-V_2-plancherel}
 to the above estimate   yields \begin{eqnarray*}\frac{\mathrm{d}}{\mathrm{d} \h{0.5pt}t} \int_0^{\infty }V_2^2 + \mathrm{R}_*^2 \int_0^{\infty} V_2^2  \h{3pt}\lesssim\h{3pt}\mu^2 \int_0^{\infty} \big| \h{1pt}\big< \h{1pt}q, i \h{0.5pt}v \h{1pt}\big> \h{1pt}\big|^2 + \mathrm{R}_*^2 \int_{\mathrm{D}_{\mathrm{R}_*}}    \Big|\h{1pt}\mathscr{F}_2 \h{1pt}\big[  V_{2} \big]\h{1.5pt} \Big|^2 \mathrm{d}\h{0.5pt}\xi.
\end{eqnarray*}
This combined  with Lemma \ref{lem-est-V_2-fourier}  shows that
 \begin{align*}  
 \frac{\mathrm{d}}{\mathrm{d} \h{0.5pt}t} \int_0^{\infty }V_2^2 + \mathrm{R}_*^2 \int_0^{\infty} V_2^2  \h{3pt}
 &\lesssim\h{3pt} \int_0^{\infty} \big| \h{1pt}\big< \h{1pt}q, i \h{0.5pt}v \h{1pt}\big> \h{1pt}\big|^2 +  \mathrm{R}_*^6 \left(  \int_0^t \| \hspace{1pt}q  \hspace{1pt}\|_{\mathrm{L}^2}\hspace{1pt}\|\hspace{1pt}v\hspace{1pt}\|_{\mathrm{L}^2}  \h{1pt}\mathrm{d} s \right)^2\\[4pt]
&\lesssim \h{2pt}  \int_0^{\infty} \left( \big| \h{1pt}\p_{\rho} \h{0.5pt}z \h{1pt}\big|^2 + \frac{\h{1pt}|\h{1pt}z \h{1pt}|^2}{\rho^2}  \right)\rho  \h{1pt} \mathrm{d}  \h{0.5pt} \rho +  \mathrm{R}_*^6 \left(  \int_0^t \| \hspace{1pt}q  \hspace{1pt}\|_{\mathrm{L}^2}\hspace{1pt}\|\hspace{1pt}v\hspace{1pt}\|_{\mathrm{L}^2}  \h{1pt}\mathrm{d} s \right)^2\\[4pt]
&=\h{2pt}  \|\h{1pt}z \h{1pt}\|_X^2 +   \mathrm{R}_*^6 \left(  \int_0^t \| \hspace{1pt}q  \hspace{1pt}\|_{\mathrm{L}^2}\hspace{1pt}\|\hspace{1pt}v\hspace{1pt}\|_{\mathrm{L}^2}  \h{1pt}\mathrm{d} s \right)^2. 
\end{align*}
In the second estimate above we have used the fact that  
$$ \big|\h{1pt}\big< \h{1pt}q, i \h{0.5pt}v \h{1pt}\big> \h{1pt}\big|\h{2pt} \lesssim\h{2pt}  \sigma^{-1} \left( \big|\h{1pt}\p_\rho z \h{1pt}\big| + \frac{|\h{1pt}z\h{1pt}|}{\rho}\right)   .   $$
The proof then follows.
\end{proof}

In light of    \eqref{eq-W*}, it turns out that $W^*\h{0.5pt}\big/\h{0.5pt}r^2$ satisfies
 \begin{equation*}
 \p_t\left( \frac{W^*}{r^2}\right)\, =\, \La_4\left( \frac{W^*}{r^2}\right) \h{2pt}+\h{2pt}   {m}\h{1pt}r^{-2}\h{1pt}\left( \p_r+r^{-1}\right)\h{1pt}  \big< \h{1pt}q, i \h{0.5pt}v\h{1pt}\big> \quad\quad\mbox{on $\R^4\times(0,\infty)$},
 \end{equation*}
 where $\La_4$ is the Laplace operator on $\R^4.$  
Then  we decompose  $W^* $ into  $W^*_{1} + W^*_{2} $,  where $W^*_{1}\h{0.5pt}\big/\h{0.5pt} r^2$ and $W^*_{2}\h{0.5pt}\big/ \h{0.5pt}r^2$ satisfy respectively the following   initial value problems:      
\begin{equation}\label{eq-W^*_1}
 \left\{
 \begin{aligned}
 \p_t\left( \frac{W^*_1}{r^2}\right) \,\,=\,\,\La_4\left( \frac{W^*_1}{r^2}\right) \h{2pt}  \quad\,\, &\mbox{in $\R^4\times(0,\infty)$};\\[3pt]
  {W^*_1}\big/ {r^2} =  {W^*_{in}}\,\big/{r^2}\qquad\quad&\mbox{at $ t=0$};
 \end{aligned}
 \right.
 \end{equation}
   and 
   \begin{equation}\label{eq-W^*_2}
 \left\{
 \begin{aligned}
 \p_t\left( \frac{W^*_2}{r^2}\right) \h{2pt}=\h{2pt} \La_4\left( \frac{W^*_2}{r^2}\right) \,+\,  {m}\h{1pt}r^{-2}\h{1pt}\left( \p_r+r^{-1}\right)\h{1pt}  \big< \h{1pt}q, i \h{0.5pt}v\h{1pt}\big> \quad &\mbox{in $\R^4\times(0,\infty)$};\\[3pt]
   {W^*_2}\big/{r^2}\, =\,0\h{172pt}\,\,&\mbox{at $ t=0$}.
        \end{aligned}
        \right.
   \end{equation}
By a standard application of the heat kernel on $\mathbb{R}^4$,   the following estimate holds for $W^*_{1}\h{0.5pt}\big/\h{0.5pt} r^2$:
   \begin{lemma}\label{lem-est-W^*_1}
For all $t > 0$, it holds 
\begin{eqnarray*}
\Big\| \h{1pt} W^*_{1}\h{1pt}\big/ r^2\big(\h{0.5pt}\cdot, t \h{0.5pt}\big) \h{1pt}\Big\|_{\mathrm{L}^{\infty}}    \,\,\lesssim\,\,t^{- 1} \h{1pt}\Big\| \h{1pt} W^*_{in}\h{1pt}\big/ r \h{1pt} \Big\|_{\mathrm{L}^2}.
\end{eqnarray*}
\end{lemma}
 Moreover we also have the following lemma concerning   the Fourier transform of $W^*_{2}\big/ r^2$ on $\mathbb{R}^4$:     
\begin{lemma}\label{lem-est-W^*_2-fourier}
For all $t > 0$ and $\eta\in \R^4$,  it holds  
  $$\Big\vert \h{1.5pt}\mathscr{F}_4\h{1pt}\big[\h{1pt}{W^*_{2}\big/ r^2}\h{1pt}\big]\big(\eta \h{1pt}\big) \h{1pt}\Big\vert \h{2pt}\lesssim\h{2pt}  |\hspace{1pt}\eta \hspace{1.5pt}| \int_0^t \| \hspace{1pt}q \hspace{1pt}\|_{\mathrm{L}^2}\hspace{1pt}\|\hspace{1pt}v\hspace{1pt}\|_{\mathrm{L}^2}  \h{1pt}\mathrm{d} s.$$ 
\end{lemma}
\begin{proof}[\bf Proof] 
Taking Fourier transform $\mathscr{F}_4$ on both sides of  \eqref{eq-W^*_2}, we have, for all $\eta \in \mathbb{R}^4$, that   
\begin{equation*}
  \p_t \hspace{0.5pt}\sF_4 \big[ \h{1pt} {W^*_{2}\big/ r^2} \h{1pt}\big]+ |\hspace{1pt}\eta\hspace{1pt}|^2 \hspace{1pt}\sF_4 \big[ \h{1pt} {W^*_{2}\big/ r^2} \h{1pt}\big]= m\h{1pt}\mathscr{F}_4 \big[  \hspace{1pt} r^{-2} \left( \p_r + r^{-1} \right) \big< \hspace{0.5pt}q, i\h{0.5pt}v  \hspace{0.5pt}\big> \hspace{1pt}\big].
\end{equation*}
For all $x\in\R^4$ and  $r=|\h{0.5pt}x\h{0.5pt}|$, it satisfies
\begin{eqnarray*}
r^{-2}\h{1pt}\big( \hspace{1pt}\p_r+r^{-1} \hspace{1pt}\big) \hspace{1pt}\big<\hspace{0.5pt} q , i\h{0.5pt}v \hspace{0.5pt}\big> 
= \nabla^{(4)} \cdot \Big(\hspace{1pt} x\hspace{1pt}r^{-3}\big<\hspace{0.5pt} q , i\h{0.5pt}v \hspace{0.5pt}\big>\hspace{1pt}\Big),
\end{eqnarray*} 
where $\D^{(4)}$ denotes the gradient operator on $\R^4$. 
 Then the above two equalities  yield  
\begin{eqnarray*}
&&\Big|\hspace{1pt}\sF_4 \left[ \h{1pt}{W^*_{2}\big/ r^2}\h{1pt}\right] \left(\eta\right)\hspace{1pt}\Big|=   \left| \hspace{2pt}m\,\int_0^t e^{-(t-s) \hspace{1pt}|\hspace{1pt}\eta\hspace{1pt}|^2} \sum_{j = 1}^4  \eta_j \hspace{1pt}\mathscr{F}_4\left[ x_j\hspace{1pt}r^{-3 } \hspace{1pt}\big<\hspace{0.5pt} q , i\h{0.5pt}v \hspace{0.5pt}\big> \hspace{1pt}\right]   \h{1pt}\mathrm{d} s  \hspace{2pt} \right|  \\[6pt] 
&&\hspace{70pt}\lesssim\hspace{2pt}  |\hspace{1pt}\eta\hspace{1.5pt}| \int_0^t  \big\| \hspace{1.5pt} x \hspace{1pt}r^{-3} \hspace{1pt}\big< \hspace{0.5pt}q , i\h{0.5pt} v  \hspace{0.5pt}\big> \hspace{1.5pt}\big\|_{\mathrm{L}^1\left(\mathbb{R}^4\right)} \h{2pt}\mathrm{d} s  
\h{2pt} \lesssim\h{2pt} |\hspace{1pt}\eta\hspace{1.5pt}| \int_0^t \int_0^{\infty} | \hspace{1pt}q  \hspace{1pt}|\hspace{1pt}|\hspace{1pt}v \hspace{1pt}| \hspace{1pt}r\hspace{2pt}\mathrm{d} r \h{1pt}\mathrm{d} s.
\end{eqnarray*}
This finishes the proof by H\"older's inequality.  
\end{proof}

 Based on Schonbek's Fourier splitting method, we have the following energy estimate  for ${W^*_{2}\h{1pt}\big/ \h{1pt}r^2}$:

\begin{lemma} \label{lem-energy-est-W^*_2}
For all $t  >0$ and $\mathrm{R}_* > 0$,   it holds  
 \begin{eqnarray*} 
 \frac{\mathrm{d}}{\mathrm{d} \h{0.5pt}t} \int_0^{\infty } \frac{\left(W^*_{2}\right)^2}{ r^2 }  + \mathrm{R}_*^2 \int_0^{\infty}\frac{\left(W^*_{2}\right)^2}{ r^2 }   \h{2pt}
 \lesssim \h{2pt}  \int_0^\infty\,\frac{\left| \h{1pt}q \h{1pt}\right|^2}{r^2}\,  +  \mathrm{R}_*^8 \left(  \int_0^t \| \hspace{1pt}q  \hspace{1pt}\|_{\mathrm{L}^2}\hspace{1pt}\|\hspace{1pt}v\hspace{1pt}\|_{\mathrm{L}^2}  \h{1pt}\mathrm{d} s \right)^2.
 \end{eqnarray*}
\end{lemma}
\begin{proof}[\bf Proof]
Similarly   as  in the proof of  Lemma \ref{lem-energy-est-V_2}, it holds
\begin{eqnarray*}
&&  \int_{\R^4} \left|\h{1pt}\D^{(4)}\left( \frac{W^*_2}{r^2}\right) \h{1pt}\right|^2 \mathrm{d} x =  \int_{\R^4} |\h{1pt}\eta\h{1pt}|^2 \h{1pt} \left|\h{1pt}\mathscr{F}_4 \h{1pt}\left[ \left( \frac{W^*_2}{r^2}\right)  \right]\h{1.5pt} \right|^2 \mathrm{d} \h{0.5pt}\eta \\[8pt]
\nonumber&&  \h{100pt} \geq \h{2pt}  \mathrm{R}_*^2 \int_{\R^4} \left( \frac{W^*_2}{r^2}\right) ^2 \h{3pt} \mathrm{d} x  - \mathrm{R}_*^2  \int_{\mathrm{B}_{\mathrm{R}_*}}    \left|\h{1pt}\mathscr{F}_4 \h{1pt} \left[   \frac{W^*_2}{r^2} \right]\h{1.5pt} \right|^2 \mathrm{d}\h{0.5pt}\eta. 
\end{eqnarray*}
Here $\mathrm{B}_{\mathrm{R}_*}\subset \R^4$ is  the ball of radius $\mathrm{R}_*$ and  center  $0.$  Denoting by $\omega_4$ the surface area of the unit sphere in $\R^4$, we have
   \begin{eqnarray}\label{eq-grad-W^*_2}
  \int_{\R^4} \left|\h{1pt}\D^{(4)}\left( \frac{W^*_2}{r^2}\right) \h{1pt}\right|^2 \mathrm{d} x\,=\, \omega_4 \int_0^{\infty} \left| \h{1pt}\p_r\left(\frac{W^*_2}{r^2}\right)\right|^2\h{1pt} r^3 \h{0.5pt} \mathrm{d}\h{0.5pt} r\, \,= \, \omega_4 \h{1pt}  \int_0^{\infty}   \frac{ \big( \p_r W^*_2 \big)^2}{r^2}.
  \end{eqnarray}
  Then the above two estimates, together with an use of  Lemma \ref{lem-est-W^*_2-fourier},  yield
\begin{eqnarray}\label{eq-W^*_2-fourier}
\begin{aligned}
  -  \int_0^{\infty} \left| \h{1pt}\p_r\left(\frac{W^*_2}{r^2}\right)\right|^2\h{1pt} r^3 \h{0.5pt} \mathrm{d}\h{0.5pt} r
\,+\,   \mathrm{R}_*^2   \int_0^{\infty}  \frac{\left(W^*_2\right)^2}{r^2}  \h{3pt}     
&\lesssim\h{3pt} \mathrm{R}_*^2  \int_{\mathrm{B}_{\mathrm{R}_*}}    \left|\h{1pt}\mathscr{F}_4 \h{1pt} \left[   \frac{W^*_2}{r^2} \right]\h{1.5pt} \right|^2 \mathrm{d}\h{0.5pt}\eta  \h{3pt}\lesssim\h{3pt}     \mathrm{R}_*^8   \left(  \int_0^t \| \hspace{1pt}q \hspace{1pt}\|_{\mathrm{L}^2}\hspace{1pt}\|\hspace{1pt}v\hspace{1pt}\|_{\mathrm{L}^2}  \h{1pt}\mathrm{d} s\right)^2. 
\end{aligned}
\end{eqnarray}Multiplying ${W^*_{2}\big/ r^2}$ on both sides of the equation \eqref{eq-W^*_2} and integrating over $\R^4$, we obtain 
\begin{equation}\label{eq-energy-W^*_2}
\frac{1}{2} \h{1pt}\frac{\mathrm{d}}{\mathrm{d} \h{0.5pt}t} \int_0^{\infty } \frac{\left(W^*_{2}\right)^2}{ r^2 }  
\,+\, \int_0^{\infty}   \left| \h{1pt}\p_r\left(\frac{W^*_2}{r^2}\right)\right|^2\h{1pt} r^3 \h{0.5pt} \mathrm{d}\h{0.5pt} r  
= - m \int_0^{\infty} \h{1pt}\p_r\left(\frac{W^*_2}{r^2}\right) \h{1pt}\big< \h{1pt}q, i \h{0.5pt}v \h{1pt}\big> .
\end{equation} 
By  Young's inequality and the fact that $|\h{0.5pt}v\h{0.5pt}|\leq1$, it follows 
\begin{equation}\label{eq-energy-est-W^*_2}
\frac{\mathrm{d}}{\mathrm{d} \h{0.5pt}t} \int_0^{\infty } \frac{\left(W^*_2\right)^2}{r^2} 
\,+\, \int_0^{\infty}   \left| \h{1pt}\p_r\left(\frac{W^*_2}{r^2}\right)\right|^2\h{1pt} r^3 \h{0.5pt} \mathrm{d}\h{0.5pt} r
 \h{2pt}\leq \h{2pt}m^2 \int_0^{\infty}   \left| \frac{1}{r}\h{1pt}\big< \h{1pt}q, i \h{0.5pt}v \h{1pt}\big> \h{1pt}\right|^2\h{1pt}
  \h{2pt} \leq  \h{2pt} m^2  \h{0.5pt}  \h{0.5pt} \int_0^\infty\,\frac{\left| \h{1pt}q \h{1pt}\right|^2}{r^2}.\,  
\end{equation} The proof then follows by adding
 \eqref{eq-W^*_2-fourier} to the estimate \eqref{eq-energy-est-W^*_2}. \end{proof}
\
\\
\\  
 \noindent \textbf{III.3. ESTIMATE FOR MODULATION PARAMETERS.}\\
   
 Now we study the modulation   parameters   $\left(\sigma\h{0.5pt},\h{0.5pt} \Theta\right)$. 
  \begin{lemma}\label{lem-est-modutaion-parameters}Suppose that Corollary \ref{cor-energy-est-uniform-bound} holds with a small positive constant $\e_*$.   Moreover we assume that $z$ is orthogonal to $\mathbf{h}_1$ in the sense of  \eqref{eq-orthogonal-z-h_1}. Then the following estimates hold  for the modulation  parameters $\left(\sigma\h{0.5pt},\h{0.5pt} \Theta\right)$:  
  \begin{equation}\label{eq-modulation-est-theta}
  \begin{aligned}
  \|\h{1pt}z \h{1pt}\|_X \left| \h{1pt}\Theta'\h{2pt}\right| \h{2pt}&\lesssim \h{2pt}   \|\h{1pt}z\h{1pt}\|_X \h{1.5pt} t^{- 1/2}  +  \| \h{1pt}z \h{1pt}\|_X^2 + \epsilon_* \h{1pt} \sigma^{-2}\h{0.5pt}\|\h{1pt}z \h{1pt}\|_X^2  +  \|\h{1pt}z\h{1pt}\|_X^2\h{1pt} \left|\h{2pt}   \frac{\sigma'}{\sigma} + \frac{\mu ^2}{m^2}  \h{2pt}\right|
   +  \epsilon_*^{-1}\int_0^{\infty} \,V^2_2  + \frac{\big( \p_r W^* \big)^2}{r^2 },
  \end{aligned}
\end{equation}
and
\begin{equation}\label{eq-modulation-est-sigma}
\left|\h{1pt}\frac{\sigma'}{\sigma} + \frac{\mu^2}{m^2} \h{1pt}\right| \h{2pt}
\lesssim  \h{2pt}  \| \h{1pt}z \h{1pt}\|_X +   \|\h{1pt}z\h{1pt}\|_X \h{1.5pt} t^{- 1/2}  +  \sigma^{-2}\h{0.5pt}\|\h{1pt}z \h{1pt}\|_X^2   +  \epsilon_*^{-1} \int_0^{\infty}  \,V^2_2  +    \frac{\big( \p_r W^* \big)^2}{r^2 }.
\end{equation}
\end{lemma}
\begin{proof}[\bf Proof] 
Taking $\mathrm{L}^2(\rho \h{1pt}\mathrm{d} \rho)$\h{0.5pt}-\h{0.5pt}inner product with
 $\mathbf{h}_{1}$ on both sides of \eqref{eq-z-Mod-HT}  and using   \eqref{eq-orthogonal-z-h_1}, we have
\begin{equation}\label{eq-equality-modulation} 
\int_0^{\infty} \mathrm{Mod} \cdot \mathbf{h}_{ 1}\h{1pt}\rho\h{1pt}\mathrm{d} \rho = - \int_0^{\infty} \mathrm{HT}\cdot \mathbf{h}_{ 1}\h{1pt}\rho\h{1pt}\mathrm{d} \rho.
\end{equation}
Here we have also used the fact that $\rL_{\mathbf{h} }[ \mathbf{h}_{1}]=0. $
 The real part of the left-hand side above is given by 
\begin{equation}\label{eq-equality-modulation-real-part} 
\begin{aligned}
 \mathrm{Re} \int_0^{\infty} \mathrm{Mod} \cdot \mathbf{h}_{1}\h{1pt}\rho\h{1pt}\mathrm{d} \rho& =  \frac{\sigma'}{\sigma} \int_0^{\infty} \mathbf{h}_{1} \h{1pt}\p_{\rho} \h{0.5pt}z_{1}\h{1pt}\rho^2 \h{1pt}\mathrm{d} \rho  - \mu^2 \int_0^{\infty} \mathbf{h}_{1} \mathbf{h}_{3}^2 \h{1pt}z_{1} \h{1pt}\rho \h{1pt}\mathrm{d} \rho\\[4pt]
& - \h{2pt} \int_0^{\infty} \Big\{ \big( 1 + \gamma  \big) \mathbf{h}_{ 1}^2 - \mathbf{h}_{ 1} \mathbf{h}_{ 3} \h{0.5pt}z_{ 2} \Big\}\left[ \Theta'  + \mu \h{0.5pt}V + \frac{m W }{r^2} \right]  \h{1pt}\rho \h{1pt}\mathrm{d} \rho.
\end{aligned}
\end{equation}
The imaginary part of the left-hand side of \eqref{eq-equality-modulation}  can be read as follows: 
\begin{equation}\label{eq-equality-modulation-imaginary-part}
\begin{aligned}
\mathrm{Im} \int_0^{\infty} \mathrm{Mod} \cdot \mathbf{h}_{ 1}\h{1pt}\rho\h{1pt}\mathrm{d} \rho =& - \int_0^{\infty} z_{ 1} \h{1pt} \mathbf{h}_{ 1} \h{0.5pt} \mathbf{h}_{ 3}  \left[  \Theta'  + \mu \h{0.5pt}V  + \frac{m W }{r^2} \right]   \rho \h{1.5pt}\mathrm{d} \h{0.5pt} \rho \\[6pt]
 & + \h{2pt}   \frac{\sigma'}{\sigma} \int_0^{\infty} \Big[ \big( 1 + \gamma  \big) m \h{0.5pt}\mathbf{h}_{ 1}^2 +  \mathbf{h}_{ 1}\h{1pt}\rho \h{1.5pt}\p_{\rho} \h{0.5pt}z_{ 2} \Big] \h{1pt}\rho \h{1.5pt}\mathrm{d} \h{0.5pt}\rho \\[6pt]
& + \h{2pt} \mu^2 \int_0^{\infty} \Big[ \big( 1 + \gamma  \big) \mathbf{h}_{ 1}^2 \h{0.5pt} \mathbf{h}_{ 3} + \mathbf{h}_{ 1}^3 \h{1pt}z_{ 2} - \mathbf{h}_{ 1} \h{0.5pt} \mathbf{h}_{ 3}^2 \h{1pt}z_{ 2} \Big] \h{1pt}\rho\h{1.5pt}\mathrm{d} \h{0.5pt}\rho.
\end{aligned}
\end{equation}

Firstly we estimate the right-hand side of \eqref{eq-equality-modulation-real-part}. Direct calculations show that 
 \begin{eqnarray*}  
 &&\frac{\sigma'}{\sigma} \int_0^{\infty} \mathbf{h}_{ 1} \h{1pt}\p_{\rho} \h{0.5pt}z_{ 1}\h{1pt}\rho^2 \h{1pt}\mathrm{d} \rho  - \mu^2 \int_0^{\infty} \mathbf{h}_{1} \mathbf{h}_{3}^2 \h{1pt}z_{1} \h{1pt}\rho \h{1pt}\mathrm{d} \rho\\[6pt]
&& \h{30pt} = \left[   \frac{\sigma'}{\sigma} + \frac{\mu^2}{m^2} \right] \int_0^{\infty} \mathbf{h}_{1} \h{1pt}\p_{\rho} \h{0.5pt}z_{1}\h{1pt}\rho^2 \h{1pt}\mathrm{d} \rho   - \mu^2 \int_0^{\infty} \Big\{   \mathbf{h}_{3}^2 \h{1pt}z_{1}  + \frac{1}{m^2} \h{1pt}\rho \h{1pt}\p_{\rho}\h{0.5pt}z_{1} \Big\}  \mathbf{h}_{1} \h{1pt}\rho \h{1.5pt} \mathrm{d} \rho.
\end{eqnarray*} 
Then  by H\"{o}lder's inequality, we can  deduce from the last equality that 
\begin{equation}\label{eq-est-real-part-first}
\left| \h{2pt} \frac{\sigma'}{\sigma} \int_0^{\infty} \mathbf{h}_{1} \h{1pt}\p_{\rho} \h{0.5pt}z_{1}\h{1pt}\rho^2 \h{1pt}\mathrm{d} \rho  - \mu^2 \int_0^{\infty} \mathbf{h}_{1} \mathbf{h}_{3}^2 \h{1pt}z_{1} \h{1pt}\rho \h{1pt}\mathrm{d} \rho \h{2pt}\right| \h{2pt}
\lesssim \h{2pt}  \|\h{1pt}z_1\h{1pt}\|_X\h{1pt}  +  \| \h{1pt}z_1 \h{1pt}\|_X \left|\h{2pt}   \frac{\sigma'}{\sigma} + \frac{\mu ^2}{m^2}  \h{2pt}\right|.
\end{equation}
Here we have used the fact that   $\rho\h{1pt}\mathbf h_1 \in \mathrm{L}^2(\rho \h{1pt}\mathrm{d} \rho)$ since $m \geq 3$. 
To estimate the last integral on the right-hand side of \eqref{eq-equality-modulation-real-part},  we split it into the sum $\mathrm{I}.1 + \mathrm{I}.2 + \mathrm{I}.3 + \mathrm{I}.4$,
 where
    \begin{eqnarray*}&& \mathrm{I}.1 = -   \Theta'  \int_0^{\infty} \Big[ \big( 1 + \gamma  \big) \mathbf{h}_{ 1}^2 - \mathbf{h}_{ 1} \mathbf{h}_{ 3}\h{1pt} z_{ 2}  \Big] \h{1pt} \rho\h{1pt}\mathrm{d} \h{0.5pt}\rho, \h{47pt} \mathrm{I}.2 = -  \int_0^{\infty} \frac{m W^{os} }{r^2} \Big[ \big( 1 + \gamma  \big) \mathbf{h}_{ 1}^2 - \mathbf{h}_{ 1} \mathbf{h}_{ 3} \h{0.5pt}z_{ 2} \Big]\h{1pt}\rho \h{1pt}\mathrm{d} \rho,\\[8pt]
&& \mathrm{I}.3 = - \int_0^{\infty} \frac{m W^*}{r^2} \h{1pt} \Big[ \big( 1 + \gamma  \big) \mathbf{h}_{ 1}^2 - \mathbf{h}_{ 1} \mathbf{h}_{ 3} \h{0.5pt}z_{ 2} \Big] \h{1pt}\rho \h{1pt}\mathrm{d} \rho,  \h{30pt} \mathrm{I}.4 = - \int_0^{\infty} \mu \h{0.5pt}V \h{1pt} \Big[ \big( 1 + \gamma  \big) \mathbf{h}_{ 1}^2 - \mathbf{h}_{ 1} \mathbf{h}_{ 3} \h{0.5pt}z_{ 2} \Big] \h{1pt}\rho \h{1pt}\mathrm{d} \rho.
\end{eqnarray*}
By   \eqref{eq-z-sobolev}, Lemma \ref{equiv-norms} and Corollary \ref{cor-energy-est-uniform-bound}, it follows that 
\begin{equation}\label{eq-z-L-infty-small}
\|\h{0.5pt} z\h{0.5pt}\|_{\mathrm L^\infty}\,\,\lesssim\,\, \|\h{0.5pt} z\h{0.5pt}\|_X \,\, \lesssim\,\,    \left\|\h{0.5pt} q\h{1pt} \right\|_{\rL^2} \,<\,\e_*.
\end{equation}
Hence we obtain  
\begin{equation}\label{eq-est-I.1}
  \left| \h{1.5pt}\Theta'\h{1.5pt}\right| \h{2pt}\lesssim\h{2pt} \left| \h{1pt}\mathrm{I}.1\h{1pt}\right|,
\end{equation}
provided that  a positive constant $\e_*   $ is suitably small. 
From \eqref{eq-W^os-bound}, it turns out  
\begin{equation}\label{eq-est-I.2}
 \big| \h{1.5pt}\mathrm{I}.2 \h{1pt}\big|   \h{2pt}\lesssim\h{2pt}  l^{-1} \int_0^{\infty} \mathbf{h}_1 \h{1pt}\rho \h{1pt}\mathrm{d} \h{0.5pt}\rho \h{2pt}\lesssim \h{2pt}  l^{-1}.
\end{equation}
In light of the following estimate 
\begin{equation}\label{eq-est- W*/r-X-partial}
 \left\| \h{1pt} \frac{W^* }{r} \h{1pt}\right\|_{\mathrm{L}^{\infty}}^2 \h{2pt}\lesssim\h{2pt}\left\| \h{1pt} \frac{W^* }{r} \h{1pt}\right\|_{X }^2 \,= \,\int_0^{\infty}\left|\p_r\left( \frac{    W^*   }{r}\right)\right|^2+ \left|\frac{ W^*  }{r^2}\right|^2\,= \,\int_0^{\infty} \frac{\big( \p_r W^*  \big)^2}{r^2},
\end{equation}
it can be shown that  
\begin{equation} \label{eq-est-I.3}
\left| \h{2pt}\mathrm{I}.3 \h{2pt}\right| \h{2pt}\lesssim\h{2pt}  \sigma^{-1} \left\| \h{1pt}\frac{W^*}{r} \h{1pt}\right\|_{\mathrm{L}^{\infty}} \int_0^{\infty} \mathbf{h}_{1} \h{2pt}\mathrm{d} \rho  
 \h{2pt}
 \lesssim\h{2pt}   \sigma^{-1} \h{1pt} \left( \int_0^{\infty} \frac{\big( \p_r W^* \big)^2}{r^2 } \right)^{1/2}.
\end{equation}
As for the integral $\mathrm{I}.4$, it  is bounded by 
 \begin{equation} \label{eq-est-I.4}
 \left| \h{2pt}\mathrm{I}.4\h{2pt}\right| \h{2pt}
\lesssim\h{2pt}    \| \h{1pt} V_1 \h{1pt}\|_{\mathrm{L}^{\infty}} \int_0^{\infty} \mathbf{h}_1 \h{1pt}\rho \h{1.5pt}\mathrm{d} \h{0.5pt}\rho +   \int_0^{\infty}  \big| \h{1pt}V_2 \h{1pt}\big| \h{1pt} \mathbf{h}_1  \h{1pt} \rho \h{1.5pt}\mathrm{d} \h{0.5pt}\rho \h{2pt}
\lesssim \h{2pt}  \| \h{1pt}V_1 \h{1pt}\|_{\mathrm{L}^{\infty}} +\sigma^{-1} \left( \int_0^{\infty} V_2^2 \right)^{1/2}.
\end{equation} 
Applying  the estimates \eqref{eq-est-real-part-first} and  \eqref{eq-est-I.1}-\eqref{eq-est-I.4} to  \eqref{eq-equality-modulation}-\eqref{eq-equality-modulation-real-part}, we get \begin{equation}\label{eq-est-modulation-theta'-first}
 \begin{aligned}
 \left| \h{1pt}\Theta'\h{1.5pt}\right| \h{2pt}
 &\lesssim \h{2pt}  l^{-1} +   \|\h{1pt}V_1 \h{1pt}\|_{\mathrm{L}^{\infty}}  +  \|\h{1pt}z_1\h{1pt}\|_X  + \|\h{1pt}z_1\h{1pt}\|_X \left|\h{2pt}   \frac{\sigma'}{\sigma} + \frac{\mu ^2}{m^2}  \h{2pt}\right|   \\[4pt]
&+ \h{2pt}  \sigma^{-1}  \left( \int_0^{\infty} V^2_2 \right)^{1/2} + \sigma^{-1} \h{1pt} \left( \int_0^{\infty} \frac{\big( \p_r W^* \big)^2}{r^2 } \right)^{1/2} + \left| \h{2pt}\int_0^{\infty} \mathrm{HT}\cdot \mathbf{h}_{1} \rho\h{1.5pt}\mathrm{d} \h{0.5pt}\rho \h{2pt}\right|.
\end{aligned}
\end{equation}
By the definition of $\mathrm{HT}$ in \eqref{eq-def-Mod-HT} and \eqref{eq-prop-gamma}, it satisfies 
\begin{equation}\label{eq-est-HT}
  \big|\h{1pt}\mathrm{HT}\h{1pt}\big| \h{2pt} \lesssim \h{2pt}  | \h{1pt}z\h{1pt}|^2 +  \sigma^{-2} \left( \h{1pt}\big| \h{1pt}\p_{\rho}\h{0.5pt}z\h{1pt}\big|^2 + \frac{ |\h{1pt}z \h{1pt}|^2}{\rho^2} \h{1pt}\right),
\end{equation}
which combined with \eqref{eq-z-sobolev}   implies 
\begin{eqnarray}\label{eq-est-HT-with-h_1}
\begin{aligned}
\left|\h{1pt}\int_0^{\infty} \mathrm{HT}\cdot \mathbf{h}_{ 1}\h{1pt}\rho\h{1pt}\mathrm{d} \rho\h{2pt}\right| \h{2pt}
&\lesssim  \h{2pt}    \int_0^{\infty} | \h{1pt}z\h{1pt}|^2 \h{1pt} \mathbf{h}_{1} \h{1pt}\rho\h{2pt}\mathrm{d} \rho + \sigma^{-2} \int_0^{\infty}  \left( \h{1pt}\big| \h{1pt}\p_{\rho}\h{0.5pt}z\h{1pt}\big|^2 + \frac{ |\h{1pt}z \h{1pt}|^2}{\rho^2} \h{1pt}\right)\mathbf{h}_{1} \h{1pt}\rho\h{2pt}\mathrm{d}\h{0.5pt} \rho \h{5pt}\\[5pt]
&\lesssim \h{3pt}  \| \h{1pt}z \h{1pt}\|_X^2 + \sigma^{-2} \h{1pt}\| \h{1pt}z\h{1pt}\|_X^2.
\end{aligned}
\end{eqnarray}
Applying this estimate to \eqref{eq-est-modulation-theta'-first}  and using Lemma \ref{lem-est-V_1}, we get \begin{eqnarray*}
&&\left| \h{1pt}\Theta'\h{2pt}\right| \h{2pt}\lesssim \h{2pt}  l^{-1} +      t^{- 1/2}  +  \|\h{1pt}z\h{1pt}\|_X  + \| \h{1pt}z \h{1pt}\|_X \left|\h{2pt}   \frac{\sigma'}{\sigma} + \frac{\mu ^2}{m^2}  \h{2pt}\right|   \\[4pt]
\nonumber &&\h{24pt} + \h{2pt}     \sigma^{-1}  \left( \int_0^{\infty} V^2_2 \right)^{1/2} + \sigma^{-1} \h{1pt} \left( \int_0^{\infty} \frac{\big( \p_r W^* \big)^2}{r^2 } \right)^{1/2}  + \sigma^{-2} \h{0.5pt}\| \h{1pt}z \h{1pt}\|_X^2.
\end{eqnarray*}
By multiplying $\| \h{1pt}z \h{1pt}\|_X$ on both sides above and utilizing  H\"{o}lder's inequality, the estimate \eqref{eq-modulation-est-theta}   follows. Here we also used  \eqref{eq-z-L-infty-small}.

Next we  establish \eqref{eq-modulation-est-sigma} by estimating each term on the right-hand side of  \eqref{eq-equality-modulation-imaginary-part}. In light of  \eqref{eq-z-sobolev},  it can be shown that 
\begin{equation}\label{eq-est-imaginary-theta'}
 \left| \h{2pt} \Theta' \h{2pt}\int_0^{\infty}   \mathbf{h}_{1} \mathbf{h}_{ 3} \h{0.5pt}z_{ 1}  \h{1pt}\rho \h{1pt}\mathrm{d} \rho \h{2pt}\right| \h{2pt}\lesssim\h{2pt}    \| \h{1pt}z_1\h{1pt}\|_{\mathrm{L}^{\infty}} \h{1pt}\big| \h{1pt}\Theta' \h{1.5pt}\big|  \int_0^{\infty} \mathbf{h}_1\h{1pt}\rho\h{1.5pt}\mathrm{d}\h{0.5pt}\rho \h{2pt}
  \lesssim \h{2pt} \|\h{1pt}z_1\h{1pt}\|_X \left| \h{1pt}\Theta'\h{2pt}\right|.
\end{equation}
 Utilizing   \eqref{eq-z-sobolev}, \eqref{eq-W^os-bound}    and \eqref{eq-est- W*/r-X-partial}, 
by similar arguments as for   \eqref{eq-est-I.2}, \eqref{eq-est-I.3}  and  \eqref{eq-est-I.4}, we have \begin{equation}\label{eq-est-imaginary-W-os}
 \left| \h{2pt}\int_0^{\infty} \frac{m W^{os} }{r^2}  \h{1.5pt}\mathbf{h}_{1} \mathbf{h}_{3} \h{0.5pt}z_{1}  \h{1pt}\rho \h{1.5pt}\mathrm{d} \h{0.5pt} \rho \h{2pt} \right| \h{2pt}
 \leq \h{2pt}\| \h{1pt}z_1 \h{1pt}\|_{\mathrm L^\infty} \h{2pt}\int_0^{\infty}  \left|  \frac{m W^{os} }{r^2}   \h{0.5pt}  \right|  \h{1.5pt}\mathbf{h}_{1}   \h{1pt}\rho \h{1.5pt}\mathrm{d} \h{0.5pt} \rho \h{2pt} \h{2pt}
 \lesssim\h{2pt}  l^{-1}\h{1pt}\| \h{1pt}z_1 \h{1pt}\|_X,
 \end{equation}
\begin{equation}\label{eq-est-imaginary-W*}
 \left| \h{2pt} \int_0^{\infty} \frac{m W ^*}{r^2} \mathbf{h}_{1}\mathbf{h}_{3} \h{1pt}z_{1} \h{1pt}\rho \h{1pt}\mathrm{d} \rho \h{2pt}\right| \h{2pt} 
 \h{2pt}\lesssim\h{2pt}   \sigma^{-1} \h{1pt}\|\h{1pt}z_1\h{1pt}\|_X \left( \int_0^{\infty} \frac{\big( \p_r W^* \big)^2}{r^2 } \right)^{1/2},
\end{equation}
and 
\begin{equation}\label{eq-est-imaginary-V}
  \left|\h{2pt}\int_0^{\infty}\mu  V  \h{1pt} \mathbf{h}_{ 1} \h{0.5pt}\mathbf{h}_{3} \h{0.5pt} z_{1} \h{1pt}\rho \h{1.5pt}\mathrm{d} \h{0.5pt} \rho \h{2pt} \right| \h{2pt}
 \lesssim\h{2pt}  \|\h{1pt} z_1 \h{1pt}\|_X  \h{1pt} \|\h{1pt}V_1\h{1pt}\|_{\mathrm{L}^{\infty}}  +  \|\h{1pt}z_1 \h{1pt}\|_X  \h{1pt}\sigma^{-1} \left( \int_0^{\infty} \, V_2^2 \right)^{1/2}.
\end{equation}
Thus by   \eqref{eq-est-imaginary-theta'}-\eqref{eq-est-imaginary-V},  the first integral on the right-hand side of \eqref{eq-equality-modulation-imaginary-part} 
 can be estimated as follows:
 \begin{equation} \label{eq-est-imaginary-theta-W-V}
 \begin{aligned}
 \left| \h{2pt} \int_0^{\infty} z_{ 1} \h{1pt} \mathbf{h}_{ 1} \h{0.5pt} \mathbf{h}_{ 3}  \left[  \Theta'  + \mu \h{0.5pt}V  + \frac{m W }{r^2} \right]   \rho \h{1.5pt}\mathrm{d} \h{0.5pt} \rho  \h{2pt}\right| \h{3pt}
 &\lesssim \h{3pt}  \|\h{1pt}z_1 \h{1pt}\|_X \left| \h{1pt}\Theta'\h{2pt}\right| +  l^{-1}\h{1pt}\| \h{1pt}z_1 \h{1pt}\|_X +   \| z_1 \|_X  \h{1pt} \|\h{1pt}V_1\h{1pt}\|_{\mathrm{L}^{\infty}}  \\[4pt]
 & + \h{2pt}  \sigma^{-1} \h{1pt}\|\h{1pt}z_1\h{1pt}\|_X \left( \int_0^{\infty} \frac{\big( \p_r W^* \big)^2}{r^2 } \right)^{1/2}  +  \|\h{1pt}z_1 \h{1pt}\|_X  \h{1pt}\sigma^{-1}  \left( \int_0^{\infty}  \,V_2^2 \right)^{1/2}.
\end{aligned}
\end{equation}
By using the   the following equality  
 \begin{eqnarray*} \int_0^{\infty} \mathbf{h}_1^2 \h{1pt}\mathbf{h}_3\h{1pt}\rho \h{1.5pt}\mathrm{d}\h{0.5pt}\rho = m^{-1} \int_0^{\infty} \mathbf{h}_1^2 \h{1pt}\rho \h{1.5pt}\mathrm{d}\h{0.5pt}\rho,
\end{eqnarray*}
the last two integrals on the right-hand side of \eqref{eq-equality-modulation-imaginary-part} can be rewritten as 
\begin{equation}\label{eq-est-imaginary-second}
\begin{aligned} 
& \left[ \frac{\sigma'}{\sigma} + \frac{\mu^2}{m^2} \right] \int_0^{\infty} \Big\{ \big( 1 + \gamma  \big) m \h{0.5pt}\mathbf{h}_{ 1}^2 +  \mathbf{h}_{ 1}\h{1pt}\rho \h{1.5pt}\p_{\rho} \h{0.5pt}z_{ 2} \Big\} \h{1pt}\rho \h{1.5pt}\mathrm{d} \h{0.5pt}\rho  \\[4pt]
& \h{30pt} + \h{2pt} \mu^2 \int_0^{\infty} \Big\{  \gamma \h{1pt}\mathbf{h}_{ 1}^2 \h{0.5pt} \mathbf{h}_{ 3} + \mathbf{h}_{ 1}^3 \h{1pt}z_{ 2} - \mathbf{h}_{ 1} \h{0.5pt} \mathbf{h}_{ 3}^2 \h{1pt}z_{ 2} \Big\} \h{1pt}\rho\h{1.5pt}\mathrm{d} \h{0.5pt}\rho - \frac{\mu^2}{m^2}  \int_0^{\infty} \Big\{  \gamma  \h{1pt} m \h{0.5pt}\mathbf{h}_{ 1}^2 +  \mathbf{h}_{ 1}\h{1pt}\rho \h{1.5pt}\p_{\rho} \h{0.5pt}z_{ 2} \Big\} \h{1pt}\rho \h{1.5pt}\mathrm{d} \h{0.5pt}\rho.
\end{aligned}
\end{equation}
In light of  the smallness of  $\| \h{1pt}z \h{1pt}\|_X$ in \eqref{eq-z-L-infty-small}, it holds \begin{equation}\label{eq-est-imaginary-second-1}
  \left|\h{1pt}\frac{\sigma'}{\sigma} + \frac{\mu^2}{m^2} \h{1pt}\right| \h{3pt}
   \lesssim\h{3pt} \left|\h{2pt} \frac{\sigma'}{\sigma} + \frac{\mu^2}{m^2} \h{2pt}\right| \h{2pt} \int_0^{\infty} \Big\{ \big( 1 + \gamma  \big) m \h{0.5pt}\mathbf{h}_{ 1}^2 +  \mathbf{h}_{ 1}\h{1pt}\rho \h{1.5pt}\p_{\rho} \h{0.5pt}z_{ 2} \Big\} \h{1pt}\rho \h{1.5pt}\mathrm{d} \h{0.5pt}\rho.
\end{equation}
By   H\"{o}lder's inequality and the fact that $\rho\h{1pt}\mathbf h_1 \in \mathrm{L}^2(\rho \h{1pt}\mathrm{d} \rho)$, the second line in \eqref{eq-est-imaginary-second} can be bounded from above by  $ \|\h{1pt}z\h{1pt}\|_X$ up to a constant depending on $m$ and $\mu$.  
This estimate combined with \eqref{eq-equality-modulation}, \eqref{eq-equality-modulation-imaginary-part} and \eqref{eq-est-imaginary-theta-W-V}-\eqref{eq-est-imaginary-second-1} then yields    \begin{eqnarray*}
   &&\left|\h{1pt}\frac{\sigma'}{\sigma} + \frac{\mu^2}{m^2} \h{1pt}\right| \h{2pt}
   \lesssim \h{2pt} \|\h{1pt}z_1\h{1pt}\|_X \left| \h{1pt}\Theta'\h{2pt}\right| +\h{2pt}  \,l^{-1}\h{1pt}\| \h{1pt}z_1 \h{1pt}\|_X  +\, \| z \|_X  \h{1pt} \|\h{1pt}V_1\h{1pt}\|_{\mathrm{L}^{\infty}}+ \, \|\h{1pt}z \h{1pt}\|_X\\[4pt]
\nonumber && \h{60pt} + \h{2pt}  \, \|\h{1pt}z \h{1pt}\|_X  \h{1pt}\sigma^{-1}  \left( \int_0^{\infty}  \,V_2^2 \right)^{1/2}+    \sigma^{-1} \h{1pt}\|\h{1pt}z\h{1pt}\|_X \left( \int_0^{\infty} \frac{\big( \p_r W^* \big)^2}{r^2 } \right)^{1/2}   +  \left| \h{2pt}\int_0^{\infty} \mathrm{HT}\cdot \mathbf{h}_{1} \rho\h{1.5pt}\mathrm{d} \h{0.5pt}\rho \h{2pt}\right|.
\end{eqnarray*} 
Applying \eqref{eq-modulation-est-theta} and  \eqref{eq-est-HT-with-h_1} to the above estimate, by  Lemma \ref{lem-est-V_1} and   H\"{o}lder's inequality, we deduce   \eqref{eq-modulation-est-sigma}. Here we also have used the smallness of $\|\h{1pt}z \h{1pt}\|_X$ in \eqref{eq-z-L-infty-small}.
\end{proof}
\
\\
\\  
 \noindent \textbf{III.4. $\mathrm{L}^2$\h{0.5pt}-\h{0.5pt}ESTIMATE OF THE VARIABLE $z$.}\\
 
 In the next lemma we derive  an energy estimate for the variable $z$.
    
  \begin{lemma}\label{lem-est-L^2-z}
  With the same assumption as in Lemma  \ref{lem-est-modutaion-parameters}, 
 we have 
 \begin{align*}
 & \frac{\mathrm{d}}{\mathrm{d}\h{0.5pt} t} \left[ \h{1pt}\sigma^2 \int_0^{\infty} |\h{1pt}z\h{1pt}|^2 \h{1 pt}\rho \h{1.5pt} \mathrm{d} \h{0.5pt} \rho \h{1pt}\right] + \mu^2\h{0.5pt} \sigma^2\int_0^{\infty} |\h{1pt}z \h{1pt}|^2 \rho\h{1.5pt}\mathrm{d}\h{0.5pt}\rho + c_* \h{1pt}  \|\h{1pt}z \h{1pt}\|_X^2 \h{3pt}
\\& \qquad \qquad \qquad   \lesssim \h{3pt}  \sigma^2 \h{1pt}\|\h{1pt}z \h{1pt}\|_X +    \sigma^2\h{1pt}\| \h{1pt}z \h{1pt}\|_X \h{1pt}t^{- 1/2}  
+  \sigma^2\h{1pt} \epsilon_*^{-1}\int_0^{\infty}  \h{1pt}V^2_2 +    \frac{ \big( \p_r W^*\big)^2}{r^2}.
\end{align*} Here $ c_* $ is a  positive constant. 
\end{lemma}
\begin{proof}[\bf Proof] 
Taking $\mathrm{L}^2(\rho\h{1.5pt}\mathrm{d}\h{0.5pt} \rho)$\h{0.5pt}-\h{0.5pt}inner product with $z$ on both sides of \eqref{eq-z-Mod-HT}, 
we obtain 
\begin{equation}\label{eq-energy-est-z-with-z}
\frac{1}{2}\h{1pt}\frac{\mathrm{d}}{\mathrm{d}\h{0.5pt} t} \int_0^{\infty} |\h{1pt}z\h{1pt}|^2 \h{1 pt}\rho \h{1.5pt} \mathrm{d} \h{0.5pt} \rho +  \sigma^{-2} \int_0^{\infty} \big|\h{1pt}\mathrm{L}_{\mathbf{h}} \h{1pt}z\h{1pt}\big|^2 \rho\h{1.5pt} \mathrm{d} \h{0.5pt}\rho = \int_0^{\infty} \big< \h{1pt}\mathrm{Mod}, z \h{0.5pt}\big> \h{1pt}\rho\h{1.5pt}\mathrm{d} \h{0.5pt}\rho + \int_0^{\infty} \big< \h{1pt}\mathrm{HT}, z \h{0.5pt}\big> \h{1pt}\rho\h{1.5pt}\mathrm{d} \h{0.5pt} \rho.
\end{equation}
Through integration by parts, the first integral on the right-hand side above can be rewritten as 
\begin{eqnarray*}
\nonumber&&\int_0^{\infty} \big< \h{1pt}\mathrm{Mod}, z \h{0.5pt}\big> \h{1pt}\rho\h{1.5pt}\mathrm{d} \h{0.5pt}\rho =   - \left[\h{1pt}\frac{\sigma'}{\sigma} + \mu^2 \h{1pt}\right] \int_0^{\infty} |\h{1pt}z \h{1pt}|^2 \rho\h{1.5pt}\mathrm{d}\h{0.5pt}\rho  -  \Theta'  \int_0^{\infty} \gamma\h{1pt}z_1\h{1pt}\mathbf{h}_1\h{1pt}\rho\h{1.5pt}\mathrm{d}\h{0.5pt}\rho \\[6pt]
\nonumber&&\h{30pt} - \h{2pt}  \int_0^{\infty} \frac{m W^{os}}{r^2} \h{1pt} ( 1 + \gamma) \h{1pt}z_1\h{1pt}\mathbf{h}_1\h{1pt}\rho\h{1.5pt}\mathrm{d} \h{0.5pt}\rho -   \int_0^{\infty} \frac{m W^*}{r^2} \h{1pt}\big(1 + \gamma\big) \h{1pt}z_1\h{1pt}\mathbf{h}_1\h{1pt}\rho\h{1.5pt}\mathrm{d} \h{0.5pt}\rho -   \int_0^{\infty} \mu \h{0.5pt} V \h{0.5pt} \big(1 + \gamma\big) \h{1pt}z_1\h{1pt}\mathbf{h}_1\h{1pt}\rho\h{1.5pt}\mathrm{d} \h{0.5pt}\rho \nonumber \\[6pt]
&& \h{30pt}   + \h{2pt}m \left[ \frac{\sigma'}{\sigma} + \frac{\mu^2}{m^2} \right]   \int_0^{\infty}  \gamma\h{0.5pt}z_2 \h{0.5pt}\mathbf{h}_1 \rho\h{1.5pt}\mathrm{d}\h{0.5pt}\rho + \mu^2 \int_0^{\infty} \big<\h{1pt}z, i \h{0.5pt}( 1 + \gamma ) \h{0.5pt}\mathbf{h}_1 \h{0.5pt}\mathbf{h}_3 + i \h{0.5pt} \mathbf{h}_1^2 \h{0.5pt} z_2 + \mathbf{h}_1^2 \h{0.5pt}z - i \h{1pt} m^{-1} \h{0.5pt}\gamma\h{1pt}\mathbf{h}_1 \h{1pt}\big> \h{1pt} \rho\h{1.5pt}\mathrm{d}\h{0.5pt}\rho.
\end{eqnarray*}
Here we have  used the orthogonal condition \eqref{eq-orthogonal-z-h_1} and the fact that $\mathbf{h}_1^2+ \mathbf{h}_3^2=1$. 
Plugging the above equality to \eqref{eq-energy-est-z-with-z} and multiplying $\sigma^2$ on both sides of the resulting equality, we get 
\begin{equation}\label{eq-energy-est-z-with-z-sigma^2}
\begin{aligned}
&\frac{1}{2}\h{1pt}\frac{\mathrm{d}}{\mathrm{d}\h{0.5pt} t} \left[ \h{1pt}\sigma^2  \int_0^{\infty} |\h{1pt}z\h{1pt}|^2 \h{1 pt}\rho \h{1.5pt} \mathrm{d} \h{0.5pt} \rho \h{1pt}\right] + \mu^2 \sigma^2  \int_0^{\infty} |\h{1pt}z\h{1pt}|^2 \h{1 pt}\rho \h{1.5pt} \mathrm{d} \h{0.5pt} \rho +   \int_0^{\infty} \big|\h{1pt}\mathrm{L}_{\mathbf{h}} \h{1pt}z\h{1pt}\big|^2 \rho\h{1.5pt} \mathrm{d} \h{0.5pt}\rho =\\[8pt] &\h{42pt}  - \h{2pt}\sigma^2 \h{1pt} \Theta'  \int_0^{\infty} \gamma\h{1pt}z_1\h{1pt}\mathbf{h}_1\h{1pt}\rho\h{1.5pt}\mathrm{d}\h{0.5pt}\rho - \sigma^2 \int_0^{\infty} \frac{m W^{os}}{r^2} \h{1pt} ( 1 + \gamma) \h{1pt}z_1\h{1pt}\mathbf{h}_1\h{1pt}\rho\h{1.5pt}\mathrm{d} \h{0.5pt}\rho - \sigma^2  \int_0^{\infty} \frac{m W^*}{r^2} \h{1pt}\big(1 + \gamma\big) \h{1pt}z_1\h{1pt}\mathbf{h}_1\h{1pt}\rho\h{1.5pt}\mathrm{d} \h{0.5pt}\rho    \\[8pt]
 & \h{42pt}   - \h{2pt} \sigma^2 \int_0^{\infty}  \h{1pt}\mu \h{1pt}  V \big(1 + \gamma\big) \h{0.5pt}z_1\h{1pt}\mathbf{h}_1\h{1pt}\rho\h{1.5pt}\mathrm{d} \h{0.5pt}\rho + m \h{1pt}\sigma^2 \left[ \frac{\sigma'}{\sigma} + \frac{\mu^2}{m^2} \right]   \int_0^{\infty}  \gamma\h{0.5pt}z_2 \h{0.5pt}\mathbf{h}_1 \rho\h{1.5pt}\mathrm{d}\h{0.5pt}\rho \\[8pt]
 &\h{42pt}   +  \h{2pt} \mu^2 \h{1pt}\sigma^2 \int_0^{\infty} \big<\h{1pt}z, i \h{0.5pt}( 1 + \gamma ) \h{0.5pt}\mathbf{h}_1 \h{0.5pt}\mathbf{h}_3 + i \h{0.5pt} \mathbf{h}_1^2 \h{0.5pt} z_2 + \mathbf{h}_1^2 \h{0.5pt}z - i \h{1pt} m^{-1} \h{0.5pt}\gamma\h{1pt}\mathbf{h}_1 \h{1pt}\big> \h{1pt} \rho\h{1.5pt}\mathrm{d}\h{0.5pt}\rho + \sigma^2\int_0^{\infty} \big< \h{1pt}\mathrm{HT}, z \h{0.5pt}\big> \h{1pt}\rho\h{1.5pt}\mathrm{d} \h{0.5pt} \rho.
\end{aligned}
\end{equation}
Now we estimate each term on the right-hand side above. Using \eqref{eq-prop-gamma} and \eqref{eq-z-sobolev}, we have 
\begin{eqnarray}\label{eq-energy-est-z-with-z-sigma^2-nonlinear} 
\left|\h{2pt}\int_0^{\infty} \gamma\h{1pt}z_1\h{1pt}\mathbf{h}_1\h{1pt}\rho\h{1.5pt}\mathrm{d}\h{0.5pt}\rho \h{2pt}\right| \h{2pt} \lesssim  \h{2pt}\|\h{1pt}z \h{1pt}\|_X^3.
\end{eqnarray}
It then follows
 \begin{eqnarray} \label{eq-energy-est-z-with-z-sigma^2-theta'} 
 \sigma^2 \left|\h{2pt} \Theta'  \int_0^{\infty} \gamma\h{1pt}z_1\h{1pt}\mathbf{h}_1\h{1pt}\rho\h{1.5pt}\mathrm{d}\h{0.5pt}\rho \h{2pt}\right| \h{2pt}\lesssim \h{2pt}  \sigma^2 \h{1pt} \| \h{1pt}z \h{1pt}\|_X^3 \h{1pt}\left|\h{1pt}\Theta'  \h{2pt}\right|.
\end{eqnarray}
In light of \eqref{eq-W^os-bound}, the second integral on the right-hand side of \eqref{eq-energy-est-z-with-z-sigma^2} can be estimated by 
\begin{eqnarray}\label{eq-energy-est-z-with-z-sigma^2-W-os} 
 \sigma^2 \left| \h{2pt}  \int_0^{\infty} \frac{m W^{os} }{r^2} \h{1pt} ( 1 + \gamma) \h{1pt}z_1\h{1pt}\mathbf{h}_1\h{1pt}\rho\h{1.5pt}\mathrm{d} \h{0.5pt}\rho  \h{2pt}\right|  \h{2pt}
 \lesssim \h{2pt} \sigma^2 \h{1pt}   \h{1pt}l^{-1}\h{0.5pt}\|\h{1pt}z \h{1pt}\|_X.
\end{eqnarray} 
Using   similar estimates as in \eqref{eq-est-I.3}-\eqref{eq-est-I.4}, we obtain   \begin{eqnarray} \label{eq-energy-est-z-with-z-sigma^2-W*} 
\begin{aligned}
\sigma^2 \left|\h{2pt}\int_0^{\infty} \frac{m W^*}{r^2} \big(1 + \gamma \h{0.5pt}\big) z_1\h{1pt}\mathbf{h}_1   \rho\h{2pt}\mathrm{d} \rho \h{2pt}\right| \h{2pt}&\lesssim\h{2pt} \sigma \h{1pt}\| \h{1pt}z \h{1pt}\|_X \left( \int_0^{\infty} \frac{\big( \p_r W^* \big)^2}{r^2} \right)^{1/2} \h{6pt}\\[8pt]
&\lesssim \h{2pt} \epsilon_* \h{1pt}\| \h{1pt}z\h{1pt}\|_X^2 + \sigma^2 \h{1pt} \epsilon_*^{-1} \int_0^{\infty} \frac{\big( \p_r W^* \big)^2}{r^2}, 
\end{aligned}
\end{eqnarray}
and 
 \begin{eqnarray}\label{eq-energy-est-z-with-z-sigma^2-V} 
 \begin{aligned}
\sigma^2 \left|  \h{1pt} \int_0^{\infty} \mu \h{1pt} V\h{0.5pt}\big(1 + \gamma\h{1pt}\big) \h{0.5pt} z_1 \h{0.5pt} \mathbf{h}_1 \h{0.5pt} \rho \h{1.5pt} \mathrm{d} \h{0.5pt}\rho \h{1pt}\right| \h{2pt}
 &\lesssim\h{2pt}  \sigma^2 \h{1pt}\| \h{1pt}z \h{1pt}\|_X\|\h{1pt} V_1\h{1pt}\|_{\mathrm{L}^{\infty}} +  \sigma  \h{1pt}  \|\h{1pt}z\h{1pt}\|_X \left( \int_0^{\infty}  \h{1pt} V^2_2 \right)^{1/2} \\[8pt] 
 &  \lesssim\h{2pt} \epsilon_*\h{1pt} \| \h{1pt}z\h{1pt}\|_X^2 +  \sigma^2  \h{1pt} \h{1pt}  \| \h{1pt}z \h{1pt}\|_X\|\h{1pt} V_1\h{1pt}\|_{\mathrm{L}^{\infty}} +   \h{1pt} \sigma^2 \h{1pt}\epsilon_*^{-1}  \int_0^{\infty}  \h{1pt}V^2_2. \h{20pt}
  \end{aligned}
\end{eqnarray}
 We are left to study the last  three integrals on the right-hand side of \eqref{eq-energy-est-z-with-z-sigma^2}. Firstly  by  a similar argument as for the estimate     \eqref{eq-energy-est-z-with-z-sigma^2-nonlinear}, we have
 \begin{eqnarray}  \label{eq-energy-est-z-with-z-sigma^2-nonlinear-1}
\left|  \h{1.5pt} m \sigma^2\left[ \frac{\sigma'}{\sigma} + \frac{\mu^2}{m^2} \right]  \int_0^{\infty}  \gamma\h{0.5pt}z_2 \h{0.5pt}\mathbf{h}_1 \rho\h{1.5pt}\mathrm{d}\h{0.5pt}\rho    \h{2pt}
\right| \,\lesssim\h{2pt} \sigma^2 \h{1pt} \|\h{1pt}z\h{1pt}\|_X^3  \left|\h{2pt}\frac{\sigma'}{\sigma} + \frac{\mu^2}{m^2}\h{2pt}\right|.
\end{eqnarray}
Moreover by    \eqref{eq-z-sobolev}, \eqref{eq-est-HT} and \eqref{eq-z-L-infty-small}, it holds 
\begin{eqnarray}  \label{eq-energy-est-z-with-z-sigma^2-nonlinear-2}
\mu^2 \h{1pt}\sigma^2 \h{1pt}\left|\h{2pt}\int_0^{\infty} \big<\h{1pt}z, i \h{0.5pt}( 1 + \gamma ) \h{0.5pt}\mathbf{h}_1 \h{0.5pt}\mathbf{h}_3 + i \h{0.5pt} \mathbf{h}_1^2 \h{0.5pt} z_2 + \mathbf{h}_1^2 \h{0.5pt}z - i \h{1pt} m^{-1} \h{0.5pt}\gamma\h{1pt}\mathbf{h}_1 \h{1pt}\big> \h{1pt} \rho\h{1.5pt}\mathrm{d}\h{0.5pt}\rho \h{2pt}\right| \h{2pt}\lesssim \h{2pt}  \h{1pt} \sigma^2 \h{1pt}\| \h{1pt}z \h{1pt}\|_X,
\end{eqnarray}
and 
\begin{eqnarray}  \label{eq-energy-est-z-with-z-sigma^2-nonlinear-3}
\begin{aligned}
\sigma^2 \left|\h{2pt} \int_0^{\infty} \big< \h{1pt}\mathrm{HT}, z \big> \h{1pt}\rho\h{1.5pt}\mathrm{d} \h{0.5pt} \rho \h{2pt}\right| \h{2pt}
&\lesssim \h{2pt} \sigma^2  \int_0^{\infty} |\h{1pt}z \h{1pt}|^3 \h{1pt}\rho \h{1.5pt}\mathrm{d} \h{0.5pt} \rho +\|\h{1pt}z\h{1pt}\|_{\mathrm L^\infty}  \int_0^{\infty}\left( | \h{1pt}\p_{\rho} \h{1pt}z \h{1pt}|^2 \h{1pt}  + \frac{|\h{1pt}z\h{1pt}|^2}{\rho^2} \right)\h{1.5pt}\rho\h{0.5pt}  \mathrm{d} \h{0.5pt}\rho \h{3pt}\\[6pt]
&\lesssim \h{3pt} \epsilon_* \h{1pt}  \sigma^2  \int_0^{\infty} |\h{1pt}z \h{1pt}|^2 \rho \h{1.5pt} \mathrm{d} \h{0.5pt} \rho +  \epsilon_* \h{1pt}\|\h{1pt}z\h{1pt}\|_X^2.
\end{aligned}
\end{eqnarray}
By applying \eqref{eq-energy-est-z-with-z-sigma^2-theta'}-\eqref{eq-energy-est-z-with-z-sigma^2-nonlinear-3}  to the right-hand side of \eqref{eq-energy-est-z-with-z-sigma^2} and employing the coercivity of $\mathrm{L}_{\mathbf{h}}$ operator given in Lemma \ref{lem-Coercivity}, it turns out 
\begin{align*}
&\frac{1}{2}\h{1pt}\frac{\mathrm{d}}{\mathrm{d}\h{0.5pt} t} \left[ \h{1pt}\sigma^2  \int_0^{\infty} |\h{1pt}z\h{1pt}|^2 \h{1 pt}\rho \h{1.5pt} \mathrm{d} \h{0.5pt} \rho \h{1pt}\right] + \frac{1}{2} \h{1pt}\mu^2 \sigma^2  \int_0^{\infty} |\h{1pt}z\h{1pt}|^2 \h{1 pt}\rho \h{1.5pt} \mathrm{d} \h{0.5pt} \rho +  c_1 \h{0.5pt}\|\h{1pt}z \h{1pt}\|_X^2   \\[8pt]
 &\qquad\quad \lesssim \h{3pt}     \sigma^2 \|\h{1pt}z \h{1pt}\|_X^3 \h{1pt}\left|\h{1pt}\Theta' \h{1pt}\right|  + \sigma^2 \h{1pt} \|\h{1pt}z \h{1pt}\|_X^3  \left|\h{2pt}\frac{\sigma'}{\sigma} + \frac{\mu^2}{m^2}\h{2pt}\right|  +  \, \sigma^2 \|\h{1pt}z \h{1pt}\|_X     \\[8pt]
  &\qquad \quad  + \h{3pt}  \sigma^2  \|\h{1pt}z\h{1pt}\|_X  \h{1pt}\Big\{\,  l^{-1}+\left\| \h{1pt}V_1 \h{1pt}\right\|_{\mathrm{L}^{\infty}} \Big\}+  \sigma^2 \h{1pt}\epsilon_*^{-1}\int_0^{\infty}  \h{1pt}V^2_2 +   \frac{\big( \p_r W^* \big)^2}{r^2}\,,
\end{align*} 
 provided  that a positive constant  $\e_*$ is suitably small. 
 Here $c_1 $ is a positive constant.  By  Lemma \ref{lem-est-V_1}, Lemma \ref{lem-est-modutaion-parameters} and the above estimate, the proof then follows.  Here we also have used the smallness of $\|\h{1pt}z \h{1pt}\|_X$ in \eqref{eq-z-L-infty-small}. 
\end{proof}
\
\\
\\
\setcounter{section}{4}
\setcounter{thm}{0}
\setcounter{equation}{0}
\noindent \textbf{IV. PROOF OF THE MAIN THEOREM.}\\

This section is devoted to the proof of our main results. 

 \begin{proof}[\bf Proof of Theorem \ref{thm-large-time-behavior}]  The local existence of \eqref{eq-main-phi-W-V} with  initial data satisfying \eqref{eq-in-orthogonal-z-h_1}-\eqref{eqq-assump-small-initial} can be obtained by methods in \cite{CY1}-\cite{CY2}. We omit it here. Moreover for the local solution obtained, denoted by $\big(W, V, \varphi\big)$, we also have the decomposition \eqref{eq-decom-varphi} for the vector field $\varphi$. In the remainder of the proof, we extend the local solution globally in time.  Suppose that $\big(W, V, \varphi\big)$ exists on the time interval $[\h{1pt}0, T\h{1pt}]$. For some $\mathrm{C}^1$-regular parameter functions $\sigma(\cdot)$ and $\Theta(\cdot)$, \eqref{eq-decom-varphi}  holds for $\varphi$ with the perturbation function $z$ satisfying \eqref{eq-orthogonal-z-h_1}. In addition we can also assume     
 \begin{equation*}
 \begin{aligned}
& (\mathrm{A}.1). \h{20pt} \big(1-\varepsilon/2\big) \h{1pt}e^{ - \frac{\mu^2}{m^2}\h{1pt}t } \sigma_{in} \h{2pt}\leq \h{2pt}\sigma\h{2pt} \leq\h{2pt}  \big(1+\varepsilon/2\big)  \h{1pt}   e^{ - \frac{\mu^2}{m^2}\h{1pt}t } \sigma_{in} , \h{20pt}&\forall\h{2pt}t \in [\h{1pt}0, T\h{1pt}];\\[8pt]
&(\mathrm{A}.2). \h{20pt} \int_0^{\infty} V_2^2 \h{2pt}\leq\h{2pt}\epsilon_*^{3/2} \h{1pt}\frac{1}{(1 + t )^2}, \h{100pt}&\forall\h{2pt}t \in [\h{1pt}0, T\h{1pt}].
\end{aligned}
\end{equation*}
Here $\varepsilon \in (0, 1)$ is an arbitrary constant.   $\e_* $ is given in Corollary \ref{cor-energy-est-uniform-bound}.  In fact by  the continuity of the parameter function $\sigma(\cdot)$, the assumption (A.1)   holds for  some $T>0$ .  Utilizing Corollary \ref{cor-energy-est-uniform-bound} and  the  estimate given below: 
$$\left\|\h{1pt}V_1\h{1pt} \right\|_{\mathrm L^2} \,\leq\,   \left\|\h{1pt}V_{in}\h{1pt} \right\|_{\mathrm L^2}\,<\, \delta_*\, <\, \e_*,$$ we get (A.2) for some $T > 0$. One should notice that the last estimate holds since $V_1$ satisfies the heat equation   \eqref{eq-for-V_1}. Now we separate the following arguments into four steps. \\
\\
\textbf{Step 1.} 
Multiplying $\exp\Big\{ \frac{\mu^2}{m^2}\h{0.5pt}t \Big\}$ on both sides of the estimate in Lemma \ref{lem-est-L^2-z} and utilizing  (A.1)-(A.2), we obtain \begin{eqnarray*} 
&&\frac{\mathrm{d}}{\mathrm{d}\h{0.5pt}t} \left[ \exp \left\{\frac{\mu^2}{m^2}\h{0.5pt}t \right\} \h{1pt} \sigma^2 \int_0^{\infty} |\h{1pt}z \h{1pt}|^2 \h{1pt}\rho\h{1.5pt}\mathrm{d}\h{0.5pt}\rho \right] + c_*  \exp \left\{\frac{\mu^2}{m^2}\h{0.5pt}t \right\} \h{1pt} \| \h{1pt}z \h{1pt}\|_X^2 \h{3pt}
\lesssim  \h{3pt} \epsilon_* \sigma_{in}^2  \exp \left\{- \frac{\mu^2}{m^2}\h{0.5pt}t \right\} \h{1pt} \\[8pt]
&& \h{60pt}+ \h{2pt}\epsilon_* \sigma_{in}^2 \h{1pt}t^{ - 1/2}  \exp \left\{- \frac{\mu^2}{m^2}\h{0.5pt}t \right\} \h{1pt}+ \epsilon_*^{1/2} \sigma_{in}^2 \h{1pt}  \exp \left\{- \frac{\mu^2}{m^2}\h{0.5pt}t \right\} \h{1pt}+ \sigma_{in}^2 \h{1pt}\epsilon_*^{-1}    \int_0^{\infty} \frac{\big( \p_r W^* \big)^2}{r^2}.
\end{eqnarray*} 
Here we also have used the smallness of  $\|\h{1pt} z \h{1pt}\|_X\h{1pt} $  in  \eqref{eq-z-L-infty-small}. Fixing a $t \in [\h{1pt}0, T\h{1pt}]$ and integrating the above estimate from $0$ to $t$, by Corollary \ref{cor-energy-est-uniform-bound}, we have
\begin{equation}\label{eq-est-energy-z-exp-1}
\exp \left\{\frac{\mu^2}{m^2}\h{0.5pt}t \right\} \h{1pt} \sigma^2 \int_0^{\infty} |\h{1pt}z \h{1pt}|^2 \h{1pt}\rho\h{1.5pt}\mathrm{d}\h{0.5pt}\rho + c_* \int_0^t  \exp \left\{\frac{\mu^2}{m^2}\h{0.5pt}s \right\} \h{1pt} \| \h{1pt}z \h{1pt}\|_X^2  \h{2pt}\mathrm{d} s \h{3pt}
\lesssim  \h{3pt} \epsilon_*^{1/2} \sigma_{in}^2 +  \sigma_{in}^2 \int_0^{\infty} \big|\h{1pt}z_{in}(\rho)  \h{1pt}\big|^2 \h{1pt}\rho\h{1.5pt}\mathrm{d}\h{0.5pt}\rho. 
\end{equation}In the next we multiply $\exp\Big\{  \frac{2\mu^2}{m^2}\h{0.5pt}t \Big\}$ on both sides of the estimate in   Lemma \ref{lem-est-L^2-z}  and employ  (A.1)-(A.2). It then follows 
 \begin{align*}
& \frac{\mathrm{d}}{\mathrm{d}\h{0.5pt}t} \left[ \exp \left\{\frac{2\mu^2}{m^2}\h{0.5pt}t \right\} \h{1pt} \sigma^2 \int_0^{\infty} |\h{1pt}z \h{1pt}|^2 \h{1pt}\rho\h{1.5pt}\mathrm{d}\h{0.5pt}\rho \right] + c_*  \exp \left\{\frac{2\mu^2}{m^2}\h{0.5pt}t \right\} \h{1pt} \| \h{1pt}z \h{1pt}\|_X^2 \h{3pt}\\[8pt]
 &\qquad\qquad\qquad\lesssim  \h{3pt} \sigma_{in}^2\h{1pt} \| \h{1pt}z \h{1pt}\|_X + \sigma_{in}^2 \h{1pt} \| \h{1pt}z \h{1pt}\|_X \h{1pt} t^{- 1/2} +   \sigma_{in}^2 \h{1pt} \epsilon_*^{1/2} \h{1pt}( 1 + t )^{-2} + \sigma_{in}^2 \h{1pt}\epsilon_*^{-1} \int_0^{\infty} \frac{\big(\p_r W^* \big)^2}{r^2}.
\end{align*}
Here we have used the assumption $m \geq 3$ so that $m^2 \geq 2$. Integrating the above estimate from $0$ to $t$ and using Corollary \ref{cor-energy-est-uniform-bound}, we obtain  
 \begin{equation}\label{eq-est-energy-z-exp-2}
 \begin{aligned}
 &\exp \left\{\frac{2\mu^2}{m^2}\h{0.5pt}t \right\} \h{1pt} \sigma^2 \int_0^{\infty} |\h{1pt}z \h{1pt}|^2 \h{1pt}\rho\h{1.5pt}\mathrm{d}\h{0.5pt}\rho + c_* \int_0^t \exp \left\{\frac{2\mu^2}{m^2}\h{0.5pt}s \right\} \h{1pt} \| \h{1pt}z \h{1pt}\|_X^2 \h{2pt}\mathrm{d} s  \h{2pt}\\[8pt]
&\h{50pt}
 \lesssim \h{3pt} \sigma_{in}^2 \h{1pt} \epsilon_*^{1/2} +  \sigma_{in}^2 \int_0^{\infty} \big|\h{1pt}z_{in}(\rho) \h{1pt}\big|^2 \h{1pt}\rho\h{1.5pt}\mathrm{d}\h{0.5pt}\rho 
 + \h{2pt} \sigma_{in}^2 \int_0^t   \| \h{1pt}z \h{1pt}\|_X \h{1pt}\mathrm{d} s + \sigma_{in}^2 \int_0^t \| \h{1pt}z \h{1pt}\|_X s^{- 1/2}\h{1pt} \mathrm{d} s .
 \end{aligned}\end{equation}
 By H\"older's inequality and  \eqref{eq-est-energy-z-exp-1}, it holds 
 \begin{equation}\label{eq-est-energy-z-X-int-time}
 \begin{aligned}
  \int_0^t   \| \h{1pt}z \h{1pt}\|_X \h{1pt}\mathrm{d} s \h{2pt} &= \h{2pt}  \int_0^t \exp \left\{- \frac{ \mu^2}{2m^2} s \right\} \exp \left\{\frac{ \mu^2}{2 m^2} s \right\} \| \h{1pt}z \h{1pt}\|_X \mathrm{d} s\\[8pt]
  &\lesssim  \h{3pt}   \left( \h{1pt}\int_0^t  \exp \left\{\frac{ \mu^2}{m^2} s\right\} \| \h{1pt}z \h{1pt}\|_X^2 \h{1pt} \mathrm{d} s \h{1pt}\right)^{1/2}\left( \h{1pt}\int_0^t  \exp \left\{-\frac{\mu^2}{m^2} s\right\} \h{1pt} \mathrm{d} s \h{1pt} \right)^{1/2} \\[8pt]
&  \lesssim  \h{3pt} \left(\sigma_{in}^2 \h{1pt}  \epsilon_*^{1/2}  +  \sigma_{in}^2 \int_0^{\infty} |\h{1pt}z_{in}(\rho) \h{1pt}|^2 \h{1pt}\rho\h{1.5pt}\mathrm{d}\h{0.5pt}\rho  \right)^{1/2}. 
  \end{aligned}
 \end{equation}
Moreover in light of \eqref{eq-z-L-infty-small},   the above estimate implies that 
    \begin{equation}\label{eq-est-energy-z-X-s^-2/1-int-time}
  \int_0^t \|\h{1pt}z \h{1pt}\|_X s^{- 1/2} \mathrm{d} s \h{2pt}\lesssim\h{2pt} 
  \left\{ 
 \begin{aligned}
 &\epsilon_*, \h{30pt}\text{if $t \leq1$;}\\[8pt]
&   \epsilon_* + \int_1^t  \|\h{1pt}z \h{1pt}\|_X  \h{1pt} \mathrm{d} s \h{2pt} \lesssim \h{3pt}\epsilon_* + \left( \sigma_{in}^2 \h{1pt} \epsilon_*^{1/2} +  \sigma_{in}^2 \int_0^{\infty} \big|\h{1pt}z_{in}(\rho) \h{1pt}\big|^2 \h{1pt}\rho\h{1.5pt}\mathrm{d}\h{0.5pt}\rho \right)^{1/2}, \h{10pt}\text{if $t > 1$}.
\end{aligned}\right.
\end{equation} 
Applying   \eqref{eq-est-energy-z-X-int-time}-\eqref{eq-est-energy-z-X-s^-2/1-int-time} to \eqref{eq-est-energy-z-exp-2}  then yields
  \begin{equation}\label{eq-est-energy-z-X-sigma^-2}
  \begin{aligned}
 & \exp \left\{\frac{2\mu^2}{m^2}\h{0.5pt}t \right\} \h{1pt} \sigma^2 \int_0^{\infty} |\h{1pt}z \h{1pt}|^2 \h{1pt}\rho\h{1.5pt}\mathrm{d}\h{0.5pt}\rho + c_* \int_0^t \exp \left\{\frac{2\mu^2}{m^2}\h{0.5pt}s \right\} \h{1pt} \| \h{1pt}z \h{1pt}\|_X^2 \h{2pt}\mathrm{d} s  \h{2pt}\\[8pt]
  &\qquad\qquad\lesssim \h{3pt} \sigma_{in}^2 \h{1pt} \left[\epsilon_*^{1/2} +    \int_0^{\infty} \big|\h{1pt}z_{in}(\rho) \h{1pt}\big|^2 \h{1pt}\rho\h{1.5pt}\mathrm{d}\h{0.5pt}\rho +   \left( \sigma_{in}^2 \h{1pt} \epsilon_*^{1/2} +  \sigma_{in}^2 \int_0^{\infty} \big|\h{1pt}z_{in}(\rho) \h{1pt}\big|^2 \h{1pt}\rho\h{1.5pt}\mathrm{d}\h{0.5pt}\rho \right)^{1/2}\right]. 
  \end{aligned}
\end{equation}
It then  follows  from   the assumption (A.1) that 
 \begin{equation}\label{eq-est-energy-z-X-sigma^-2-int-time}
  \int_0^t\, \sigma^{-2} \h{1pt}\|\h{1pt}z \h{1pt}\|_X^2\h{2pt}\mathrm{d} s \h{3pt}
  \lesssim  \h{3pt}\epsilon_*^{1/2} +   \int_0^{\infty} \big|\h{1pt}z_{in}(\rho) \h{1pt}\big|^2 \h{1pt}\rho\h{1.5pt}\mathrm{d}\h{0.5pt}\rho +   \left( \sigma_{in}^2 \h{1pt} \epsilon_*^{1/2} +  \sigma_{in}^2  \int_0^{\infty} \big|\h{1pt}z_{in}(\rho) \h{1pt}\big|^2 \h{1pt}\rho\h{1.5pt}\mathrm{d}\h{0.5pt}\rho \right)^{1/2}. 
\end{equation}\\
\textbf{Step 2.}  \h{3pt}
Integrating  the estimate \eqref{eq-modulation-est-sigma}   from $0$ to $t$,  by \eqref{eq-est-energy-z-X-int-time}-\eqref{eq-est-energy-z-X-sigma^-2-int-time}, Corollary \ref{cor-energy-est-uniform-bound} and the assumption (A.2), we obtain   
 \begin{equation}\label{eq-est-sigma-final}
 \int_0^t \left| \h{1pt}\frac{\sigma'}{\sigma} + \frac{\mu^2}{m^2} \h{1pt}\right| \h{1pt}\mathrm{d} s\h{3pt}
 \lesssim \h{3pt} \epsilon_*^{1/2} + \int_0^{\infty} \big|\h{1pt}z_{in}(\rho) \h{1pt}\big|^2 \rho \h{1.5pt}\mathrm{d}\h{0.5pt}\rho  +  \left( \sigma_{in}^2 \h{1pt} \epsilon_*^{1/2} +  \sigma_{in}^2 \int_0^{\infty} \big|\h{1pt}z_{in}(\rho) \h{1pt}\big|^2 \h{1pt}\rho\h{1.5pt}\mathrm{d}\h{0.5pt}\rho \right)^{1/2}. 
\end{equation}
Since it holds 
\begin{eqnarray*} 
 \left| \h{2pt} \ln \left(\frac{\sigma}{\sigma_{in}}\,\, e^{ \frac{\mu^2}{m^2} t} \right)  \h{1pt}\right| \,=\,  \left| \h{1pt} \ln \sigma + \frac{\mu^2}{m^2} t - \ln \sigma_{in}  \h{1pt}\right| \h{2pt}\leq\h{2pt}\int_0^t \left| \h{1pt}\frac{\sigma'}{\sigma} + \frac{\mu^2}{m^2} \h{1pt}\right| \h{2pt}\mathrm{d} s, 
\end{eqnarray*}we then have, from  the above two estimates, that
\begin{equation}\label{eq-est-sigma-exponential-decay}
\left(1-\varepsilon/4\right)\h{1pt}  e^{ - \frac{\mu^2}{m^2}\h{1pt}t }\h{1pt} \sigma_{in}\h{2pt}\leq \h{2pt}\sigma\h{2pt} \leq\h{2pt} \left(1+\varepsilon/4\right) \h{1pt}e^{ - \frac{\mu^2}{m^2}\h{1pt}t } \h{1pt}\sigma_{in}, \h{20pt}\forall\h{2pt}t \in [\h{1pt}0, T \h{1pt}].
\end{equation}Here we need $\epsilon_*$ and $\delta_* $ in \eqref{eqq-assump-small-initial} to be sufficiently small. The smallness depends  on $m$, $\mu$, $\omega$, $r_0$ and $\varepsilon$.  The assumption (A.1) is then improved.
\\
\\
\textbf{Step 3.} \h{3pt} Now we improve (A.2).  
By Lemma \ref{lem-energy-est-V_2}  and Corollary  \ref{cor-energy-est-uniform-bound},   it holds
 \begin{equation} \label{eq-energy-est-V_2-improve} 
 \frac{\mathrm{d}}{\mathrm{d} \h{0.5pt}t} \int_0^{\infty }V_2^2 + \mathrm{R}_*^2 \int_0^{\infty} V_2^2  \h{3pt}
 \lesssim\h{3pt}  \|\h{1pt}z \h{1pt}\|_X^2 +   \h{1pt}\mathrm{R}_*^6  \h{1pt} \e_*^2 \h{0.5pt}  \left(  \int_0^t  \hspace{1pt}\|\hspace{1pt}v\hspace{1pt}\|_{\mathrm{L}^2}  \h{1pt}\mathrm{d} s \right)^2, \h{20pt} \text{for all $t \in [\h{1pt}0, T\h{1pt}]$ and $\mathrm{R}_* > 0$.}
 \end{equation} 
In light of  \eqref{eq-est-L2-v-h_1-z},     (A.1) and  \eqref{eq-est-energy-z-X-sigma^-2},  the $\mathrm{L}^2$-norm of $v$ satisfies \begin{equation}\label{eq-est-decay-L^2-v} 
\begin{aligned}
  \int_0^{\infty}  | \h{1pt}v  \h{1pt}|^2 \,\lesssim\,  \sigma^2  +\sigma^2   \int_0^{\infty}  \big|\h{1pt}z(\rho)\h{1pt}\big|^2 \h{1pt}\rho\h{1pt} \mathrm{d}\h{1pt} \rho\, \lesssim \,   \sigma_{in}^2\h{1pt}\exp\left\{ - \frac{2\mu^2}{m^2} t \right\}, \h{20pt}\text{for all $t \in [\h{1pt}0, T\h{1pt}]$.}
  \end{aligned}
\end{equation}This estimate then implies
\begin{equation}\label{eq-est-decay-L^2-v-int-in-time}
   \int_0^t  \hspace{1pt}\|\hspace{1pt}v\hspace{1pt}\|_{\mathrm{L}^2}  \h{1pt}\mathrm{d} s   \,\lesssim \,   \sigma_{in} \h{1pt} \int_0^t \exp\left\{ - \frac{\mu^2}{m^2} s \right\}  \h{1pt}\mathrm{d} s \,\lesssim \,\,   \sigma_{in} . 
\end{equation}
Plugging this estimate into the right-hand side of   \eqref{eq-energy-est-V_2-improve}, we get  
 \begin{equation*}  
 \frac{\mathrm{d}}{\mathrm{d} \h{0.5pt}t} \int_0^{\infty }V_2^2 + \mathrm{R}_*^2 \int_0^{\infty} V_2^2  \h{3pt}
 \lesssim  \h{3pt}   \|\h{1pt}z \h{1pt}\|_X^2 +   \h{1pt}\mathrm{R}_*^6  \h{1pt} \e_*^2 \h{2pt}  \sigma_{in}^2.
 \end{equation*}
 By taking $\mathrm{R}_*^2 = 3 ( 1 + t )^{-1}$ in the above estimate, it turns out\begin{eqnarray*}
 \frac{\mathrm{d}}{\mathrm{d}\h{0.5pt}t} \int_0^{\infty} V_2^2 + 3 ( 1 + t )^{-1} \int_0^{\infty} V_2^2 \h{3pt}
  \lesssim  \h{3pt}  \|\h{1pt}z \h{1pt}\|_X^2 + (1 + t)^{-3}\h{1pt} \epsilon_*^2 \h{1pt} \sigma_{in}^2.
\end{eqnarray*}
Multiplying $(1 + t)^3$ on both sides above and integrating from $0$ to $t$,  we have, with an use of  \eqref{eq-est-energy-z-exp-1}, that
\begin{align*}
(1 + t)^3 \int_0^{\infty} V_2^2 \h{3pt}
  &\lesssim \h{3pt} \epsilon_*^2\, \sigma_{in}^2\, t + \int_0^t   \| \h{1pt}z \h{1pt}\|_X^2 \h{1pt}( 1 + s )^3  \h{1pt}\mathrm{d} s \h{2pt}\\[6pt]
   &\lesssim \h{3pt}  \epsilon_*^2\, \sigma_{in}^2\h{2pt}t + \h{2pt} \epsilon_*^{1/2} \h{1pt} \sigma_{in}^2 +  \sigma_{in}^2 \int_0^{\infty} \big|\h{1pt}z_{in}(\rho) \h{1pt}\big|^2 \h{1pt}\rho\h{1pt}\mathrm{d}\h{0.5pt}\rho.
\end{align*}
One should notice that  as defined  in \eqref{eq-for-V_2}, the initial value of $ V_2 $ equals to $0$. This fact is also applied in the last estimate.   Now we choose   $\epsilon_*$ and $\delta_* $ in \eqref{eqq-assump-small-initial} to be sufficiently small. The smallness depends on  $m$, $\mu$, $\omega$, $r_0$ and $\varepsilon$. The above estimate immediately implies
\begin{equation}\label{eq-est-decay-V_1-improved}
 \int_0^{\infty} V_2^2 \,\, \leq \,\, \frac{\e_*^2}{(1+t)^2},\h{20pt}\forall\h{2pt}t \in [\h{1pt}0, T \h{1pt}].
\end{equation}
This  improves the assumption (A.2). By the above arguments, we can extend the local solution $\big(W, V, \varphi\big)$ globally in time.  Furthermore by taking $t \rightarrow \infty$, the estimate \eqref{eq-int-L^2-q} follows from \eqref{eq-est-energy-z-X-sigma^-2}. \\
\\
\textbf{Step 4.}\h{3pt} In this last step we give a time decay estimate for the $\mathrm{L}^2$-norm of $W^*_2\big/ r$. Let
\begin{equation*} 
 \mathrm{E}^*_2\, :=\,  \int_0^{\infty}   |\h{1.5pt}q\h{1.5pt}|^2 + V_2^2 + \frac{\big( W_2^* \big)^2}{r^2}. 
\end{equation*}
Adding the equalities 
\eqref{eq-energy-est-q-1}, \eqref{eq-energy-V_2} and \eqref{eq-energy-W^*_2}, by   \eqref{eq-grad-W^*_2},  we have 
 \begin{equation}\label{energy-e2} 
 \begin{aligned}
 \frac{1}{2}\h{1pt}\frac{\mathrm{d}}{\mathrm{d} \h{1pt}t} \h{1.5pt} \mathrm{E}^*_2 + \int_0^{\infty}  \big| \h{1pt}\mathrm{L}_{m} q \h{1pt}\big|^2 + \big( \p_r V_2 \big)^2 + \frac{ \big( \p_r W^*_2 \big)^2}{r^2} &= - \mu^2 \int_0^{\infty} \big< \h{1pt}\mathrm{L}_{m} q, \varphi_{3} \h{0.5pt}v \h{1pt}\big>\\[8pt]
 &  \quad -   \int_0^{\infty}\left[  \frac{m\h{0.5pt} W^{os} }{r^2} + \frac{ m\h{0.5pt}W^*_1}{r^2} +\mu V_1\right] \h{1pt} \big< \h{0.5pt} \mathrm{L}_{m} q, i \h{0.5pt}v \h{0.5pt}\big> .
 \end{aligned}
\end{equation}
Here we have used 
  integration by parts to obtain the last integral above (refer to \eqref{eq-integration-by-part-Lm} and \eqref{eq-est-W-os-Lm-q}).  In light of  the bound \eqref{eq-W^os-bound} and  Lemmas \ref{lem-est-V_1}, \ref{lem-est-W^*_1},  for all $t > 0$, it holds 
\begin{eqnarray*}
\left\| \h{1pt} \frac{m\h{0.5pt} W^{os} }{r^2} + \frac{ m\h{0.5pt}W^*_1}{r^2} +\mu V_1 \h{1pt}\right\|_{\mathrm{L}^{\infty}}    \,\,
\lesssim \,\,  \h{2pt}  t^{- 1 / 2 } \h{2pt}\big\| \h{1pt}V_{in} \h{1pt} \big\|_{\mathrm{L}^2} +  t^{- 1} \h{1pt}\left[\h{0.5pt}  1 +\left\| \h{1pt} \frac{W^*_{in}}{r} \h{1pt} \right\|_{\mathrm{L}^2}\right].
\end{eqnarray*} Applying this estimate and Young's inequality to the right-hand side of \eqref{energy-e2}, we obtain  \begin{eqnarray}\label{eq-energy-est-E^*_2}
 \begin{aligned}
 \h{1pt}\frac{\mathrm{d}}{\mathrm{d} \h{1pt}t} \h{1.5pt} \mathrm{E}^*_2 + \int_0^{\infty}  \big| \h{1pt}\mathrm{L}_{m} q \h{1pt}\big|^2 + \big( \p_r V_2 \big)^2 + \frac{ \big( \p_r W^*_2 \big)^2}{r^2}\, \,\lesssim\,   \,  \left( 1 \h{1pt} +\h{1pt}  t^{- 2} \right) \h{1pt}\int_0^{\infty} |\h{1pt}v \h{1pt}|^2.
 \end{aligned}
\end{eqnarray}  
By \eqref{eq-z-L-infty-small},  the coercivity estimate in Lemma  \ref{eq-coer-L_m}  holds. Thus  Lemmas \ref{lem-energy-est-W^*_2} and \ref{eq-coer-L_m} assert  that 
  \begin{eqnarray*} 
 \frac{\mathrm{d}}{\mathrm{d} \h{0.5pt}t} \int_0^{\infty } \frac{\left(W^*_{2}\right)^2}{ r^2 }  + \mathrm{R}_*^2 \int_0^{\infty}\frac{\left(W^*_{2}\right)^2}{ r^2 }   \h{3pt}
  \lesssim   \h{3pt}   \int_0^{\infty}  \big| \h{1pt}\mathrm{L}_{m} q \h{1pt}\big|^2 +   \h{1pt}\mathrm{R}_*^8 \left(  \int_0^t \| \hspace{1pt}q  \hspace{1pt}\|_{\mathrm{L}^2}\hspace{1pt}\|\hspace{1pt}v\hspace{1pt}\|_{\mathrm{L}^2}  \h{1pt}\mathrm{d} s \right)^2, \h{20pt}\text{for all $t > 0$ and $\mathrm{R}_* > 0$.}
 \end{eqnarray*}
 Multiplying a small positive constant  $c_1 $ on both sides above and summing 
the resulting estimate with the estimate  \eqref{eq-energy-est-E^*_2}, we get  
  \begin{eqnarray*} 
  \begin{aligned}
&\big(1+c_1\big) \,\frac{\mathrm{d}}{\mathrm{d} \h{0.5pt}t} \int_0^{\infty } \frac{\left(W^*_{2}\right)^2}{ r^2 }  \h{1pt}+\h{1pt} c_1\h{1pt} \mathrm{R}_*^2 \int_0^{\infty}\frac{\left(W^*_{2}\right)^2}{ r^2 }  \,+\,\frac{\mathrm{d}}{\mathrm{d} \h{0.5pt}t}  \int_0^{\infty}   |\h{1.5pt}q\h{1.5pt}|^2 + V_2^2  \h{1pt} \h{2pt}
\\[8pt]
 &\qquad\qquad\qquad\qquad\qquad\lesssim  \h{3pt}  \big( 1 \h{1pt} +\h{1pt}  t^{- 2}\big) \h{1pt}\int_0^{\infty} |\h{1pt}v \h{1pt}|^2 
\,+  \, \mathrm{R}_*^8 \,\left(  \int_0^t \| \hspace{1pt}q  \hspace{1pt}\|_{\mathrm{L}^2}\hspace{1pt}\|\hspace{1pt}v\hspace{1pt}\|_{\mathrm{L}^2}  \h{1pt} \mathrm{d} s \right)^2.  
 \end{aligned}
 \end{eqnarray*}
By  \eqref{eq-est-decay-L^2-v}-\eqref{eq-est-decay-L^2-v-int-in-time}   and Corollary  \ref{cor-energy-est-uniform-bound}, this estimate implies  \begin{eqnarray*} 
  \begin{aligned}
&  \,\frac{\mathrm{d}}{\mathrm{d} \h{0.5pt}t} \int_0^{\infty } \frac{\left(W^*_{2}\right)^2}{ r^2 }  + c_2\h{1pt} \mathrm{R}_*^2 \int_0^{\infty}\frac{\left(W^*_{2}\right)^2}{ r^2 }  \,+\, c_3\h{1pt}\frac{\mathrm{d}}{\mathrm{d} \h{0.5pt}t}  \int_0^{\infty}   |\h{1.5pt}q\h{1.5pt}|^2 + V_2^2   \h{3pt}\lesssim  \h{3pt}  \big( 1 \h{1pt} +\h{1pt}  t^{- 2}\big) \h{1pt} \sigma_{in}^2\h{1pt}\exp\left\{ - \frac{2\mu^2}{m^2} t \right\}
\,+  \, \mathrm{R}_*^8  \h{1.5pt}\e_*^2 \h{1.5pt}\sigma_{in}^2,
 \end{aligned}
 \end{eqnarray*}where $c_2 $ and $c_3 $ are two positive constants.  Now we take $c_2 \h{1pt}\mathrm{R}_*^2 = 3\h{1pt}t^{-1}$ and multiply $t^3$ on both sides above. It then turns out  
   \begin{eqnarray*} 
  \begin{aligned}
 \frac{\mathrm{d}}{\mathrm{d} \h{0.5pt}t} \left[ t^3\int_0^{\infty } \frac{\left(W^*_{2}\right)^2}{ r^2 }  \right]  \h{2pt}+ c_3\h{1pt} t^3\h{1pt}\frac{\mathrm{d}}{\mathrm{d} \h{0.5pt}t}  \int_0^{\infty}   |\h{1.5pt}q\h{1.5pt}|^2 + V_2^2  \h{1pt} \h{3pt}
 &\lesssim  \h{3pt}  \sigma_{in}^2\h{1pt} \big( t \h{1pt} + \h{1pt}  t^{ 3} \h{1pt}\big) \h{1pt}\exp\left\{ - \frac{2\mu^2}{m^2} t \h{1pt}\right\}  +   \h{1pt} t^{-1}\h{1.5pt} \e_*^2 \h{1.5pt}  \sigma_{in}^2 . 
 \end{aligned}
 \end{eqnarray*}
By integrating the above estimate  from $1$ to $t$, it follows, for all $t>1$, that
     \begin{equation}\label{eq-est-W^*_2-splitting-2}
  \begin{aligned}
  t^3\int_0^{\infty } \frac{\left(W^*_{2}\right)^2}{ r^2 }  \,+ \,c_3\int_1^t  \h{1pt} s^3\h{1pt}\left( \frac{\mathrm{d}}{\mathrm{d} \h{0.5pt}s} \int_0^{\infty}   |\h{1pt}q\h{1pt}|^2 + V_2^2 \right)\h{1pt} {\mathrm{d} \h{0.5pt}s}  \h{3pt}
 &\lesssim   \h{3pt}\mathrm E^*_2(1) \,+\, 1+  \h{1pt} t .
  \end{aligned}
 \end{equation}
For the second integral on the left-hand side    above, we have, through integration by parts, that 
    \begin{equation*} 
  \begin{aligned}
 \int_1^t  \h{1pt} s^3\h{1pt}\left( \frac{\mathrm{d}}{\mathrm{d} \h{0.5pt}s} \int_0^{\infty}   |\h{1pt}q\h{1pt}|^2 + V_2^2 \right)\h{1pt} {\mathrm{d} \h{0.5pt}s} 
 \,\,&\geq  \,\, t^3    \int_0^{\infty}   |\h{1pt}q\h{1pt}|^2 + V_2^2 \h{1pt}  -\,\mathrm E^*_2(1)\,  - 3 \int_1^t  \h{1pt} s^2\h{1pt}   \left(\int_0^{\infty}   |\h{1pt}q\h{1pt}|^2 + V_2^2\right)\h{1pt} {\mathrm{d} \h{0.5pt}s}\\[8pt]
&\gtrsim  \,\,-\,\mathrm E^*_2(1)\,-\,  \int_1^t  \h{1pt} s^2\h{1pt} \left\| \hspace{1pt}q  \hspace{1pt}\right\|_{\mathrm{L}^2}^2 \h{1pt} {\mathrm{d} \h{0.5pt}s}\,-\, \int_1^t  \h{1pt} s^2\h{1pt}    \left\| \hspace{1pt} V_2  \hspace{1pt}\right\|_{\mathrm{L}^2}^2 \h{1pt} {\mathrm{d} \h{0.5pt}s}, \h{20pt}\forall\h{3pt}t > 1.
 \end{aligned}
 \end{equation*} 
Since (A.2) holds for all $t > 0$, this fact together with   \eqref{eq-int-L^2-q}, Lemma \ref{equiv-norms},  and the above estimate,  gives us
    \begin{equation*} 
  \begin{aligned}
 \int_1^t  \h{1pt} s^3\h{1pt}\left( \frac{\mathrm{d}}{\mathrm{d} \h{0.5pt}s} \int_0^{\infty}   |\h{1pt}q\h{1pt}|^2 + V_2^2 \right)\h{1pt} {\mathrm{d} \h{0.5pt}s} \,\,
\gtrsim  \,\,- \,\Big[\,\mathrm E^*_2\h{1pt}(1)\,+\, 1 +\, t\,\Big].
 \end{aligned}
 \end{equation*}
Plugging  this estimate into \eqref{eq-est-W^*_2-splitting-2} then yields   
    \begin{equation}\label{eq-est-W^*_2-splitting-last}
  \begin{aligned}
  t^3\int_0^{\infty } \frac{\left(W^*_{2}\right)^2}{ r^2 }  \,\,\lesssim    \,\,\mathrm E^*_2(1) \,+\, 1+  \h{1pt} t \qquad \hbox{for all $t>1$}.
  \end{aligned}
 \end{equation}
 Since $V_1$ and $W^*_1$ solve  \eqref{eq-for-V_1} and \eqref{eq-W^*_1}, respectively, it satisfies 
 $$ \int_0^{\infty}   V_1 ^2\, \leq\, \int_0^{\infty}   V_{in}^2  \quad \hbox{and}\quad  \int_0^{\infty}    \frac{\big( W_1^* \big)^2}{r^2}\,\leq\,\int_0^{\infty}    \frac{\big( W_{in}^* \big)^2}{r^2}. $$
These estimates combined with  Lemma \ref{lem-energy-est-uniform-bound}  show that     
 $$\mathrm E^*_2(t) \,\lesssim \,\mathrm E^*(t) + \int_0^{\infty}   V_1 ^2 + \frac{\big( W_1^* \big)^2}{r^2} \,\,\lesssim \, \,\mathrm E(0) \quad \hbox{for all $t>0$}. $$
Applying  this estimate to \eqref{eq-est-W^*_2-splitting-last} yields  the last estimate in \eqref{eq-est-time-decay}. The proof is finished.  \end{proof}

 In the end we prove Corollary \ref{cor-main-result}.

\begin{proof}[\bf Proof of Corollary \ref{cor-main-result}] 
By the proof of Theorem \ref{thm-large-time-behavior}, it remains to obtain the convergence of the  rotation parameter $\Theta$ as $t\to\infty$.  Using  the same derivations as for \eqref{eq-est-modulation-theta'-first} and employing \eqref{eq-est-HT-with-h_1}, we have   
\begin{equation*}
 \left| \h{1pt}\Theta'\h{2pt}\right| \h{2pt}\lesssim\h{2pt}    \,\|\h{1pt}z\h{1pt}\|_X  + \| \h{1pt}z \h{1pt}\|_X \left|\h{2pt}   \frac{\sigma'}{\sigma} + \frac{\mu ^2}{m^2}  \h{2pt}\right|  + \sigma^{-2} \h{0.5pt}\| \h{1pt}z \h{1pt}\|_X^2.
\end{equation*}Here one just needs to take  $V\equiv0$, $W^*\equiv0$  and $\omega=0$  in the proof of \eqref{eq-est-modulation-theta'-first}.
Integrating the above estimate from $0$ to $\infty$, by  the estimates in  \eqref{eq-est-energy-z-X-int-time}, \eqref{eq-est-energy-z-X-sigma^-2-int-time} and   \eqref{eq-est-sigma-final}, we can show that  
\begin{equation*}
\int_0^\infty\, \left| \h{1pt}\Theta'\h{2pt}\right| \,\mathrm d s \h{2pt} \lesssim  \h{3pt}   \epsilon_*^{1/2} + \int_0^{\infty} |\h{1pt}z_{in}(\rho) \h{1pt}|^2 \rho \h{1.5pt}\mathrm{d}\h{0.5pt}\rho  +  \left( \sigma_{in}^2 \h{1pt} \epsilon_*^{1/2} +  \sigma_{in}^2 \int_0^{\infty} |\h{1pt}z_{in} (\rho)\h{1pt}|^2 \h{1pt}\rho\h{1.5pt}\mathrm{d}\h{0.5pt}\rho \right)^{1/2}.
\end{equation*} 
This implies that  $\Theta(t) $ converges to some constant  $\Theta_\infty$ as $t\to\infty$. The proof is finished.
\end{proof} \
 
\noindent{\bf Conflict of interest:} WE DECLARE THAT THERE IS NO CONFLICT OF INTERESTS.


  \end{document}